\documentclass[bj]{imsart}
\usepackage{amsthm,amsmath,natbib}
\usepackage{mathtools}
\usepackage{amssymb}
\usepackage{environ}
\usepackage{graphicx}

\usepackage{floatrow}

\usepackage{algorithm}
\usepackage{algorithmic}

\floatstyle{plaintop}
\restylefloat{table}
\usepackage{bm}
\RequirePackage[colorlinks,citecolor=blue,urlcolor=blue]{hyperref}

\startlocaldefs
\newfloatcommand{capbtabbox}{table}[][\FBwidth]
\newcounter{dummy} \numberwithin{dummy}{section}

\newcommand{\mc}{\mathcal}
\newcommand{\mb}{\mathbb}
\newtheorem{prop}[dummy]{Proposition}
\newtheorem{lemma}[dummy]{Lemma}
\newtheorem{coro}[dummy]{Corollary}
\newtheorem{defin}[dummy]{Definition}

\newtheorem{example}[dummy]{Example}
\newtheorem{remark}[dummy]{Remark}

\newtheorem{setup}[dummy]{Setup}
\newcommand*{\rom}[1]{\expandafter\@slowromancap\romannumeral #1@}
\DeclareMathOperator{\Tr}{tr}

\DeclarePairedDelimiter\floor{\lfloor}{\rfloor}

\endlocaldefs

\begin{document}
%\doublespacing
%\begin{spacing}{2}

\bibliographystyle{imsart-nameyear}

\begin{frontmatter}

% "Title of the paper"
\title{New Tests of Uniformity on the Compact Classical Groups as Diagnostics for Weak-$^*$ Mixing of Markov Chains}
\runtitle{New tests of uniformity on the compact classical groups}

% indicate corresponding author with \corref{}
\begin{aug}
 \author{\fnms{Amir} \snm{Sepehri}\corref{}\ead[label=e1]{asepehri@stanford.edu}\thanksref{t1}}
\thankstext{t1}{This work is part of author's PhD dissertation at Stanford University} 
 \address{Department of Statistics\\ Sequoia Hall, Stanford, CA 94305, USA\\ \printead{e1}}
 \affiliation{Stanford University}
%\and
%\author{\fnms{???} \snm{???}\ead[label=e2]{???}}
%\address{\printead{e2}}
%\affiliation{???}

\runauthor{Amir Sepehri}
\end{aug}

\begin{abstract}
This paper introduces two new families of non-parametric tests of goodness-of-fit on the compact classical groups. 
One of them is a family of tests for the eigenvalue distribution induced by the uniform distribution, which is 
consistent against all fixed alternatives. The other is a family of tests 
for the uniform distribution on the entire group, which is again consistent against all fixed alternatives. The construction of these tests heavily employs facts and techniques from the representation theory of compact groups. In particular, new Cauchy identities are derived and proved for the characters of compact classical groups, in order to accomodate the computation of the test statistic.
We find the 
asymptotic distribution under the null and general alternatives. The tests are proved to be asymptotically admissible. Local power is derived and the global properties of the power function against local alternatives are explored.

The new tests are validated on two random walks for which the mixing-time is studied in the literature. The new tests, and several others, are applied to the Markov chain sampler proposed by \cite{jones2011randomized}, providing strong evidence supporting the claim that the sampler mixes quickly.
\end{abstract}


%\tableofcontents
\end{frontmatter}

\section{Introduction}\label{Sec:Intro}
Recent work of \citet{jones2011randomized} suggested a Markov chain on the orthogonal group that is supposedly used to sample from the uniform distribution. They prescribe a particular number of steps after which the chain is mixed, resulting in a fast random rotation generator which is at the core of several successful randomized data analysis algorithms. Examples include approximate algorithms for highly over-determined linear regression \citep{rokhlin2008fast}, low-rank matrix approximation \citep{liberty2007randomized}, and very high dimensional nearest neighbor analysis \citep{jones2011randomized}. The new sampler could offer a significant reduction in computational cost compared to the best exact algorithm in the literature (see section \ref{Sec:RokhGen}). In applications where multiplication of a random orthogonal matrix with many vectors is needed, the new sampler is much faster than conventional random rotation generators.

It is desirable to have outputs that are approximately uniformly distributed. This is not just a mere theoretical preference; it is a matter of practical importance. In fact, as discussed in Observation 5.1 in the supplementary material \cite{sepehri2017supplement}, the performance of the approximate nearest neighbor algorithm was improved by using a uniform sampler compared to non-uniform samplers (the approximate nearest neighbor algorithm is sketched in section 3 of the supplementary material \cite{sepehri2017supplement}). Therefore, one needs to investigate the mixing properties of the new sampler.
Unfortunately, due to complex construction of the new sampler, analytical study of the mixing-time seems to be impractical. This paper suggests to numerically study the mixing-time of the new sampler using statistical tests of goodness-of-fit.

There is a sizable literature on goodness-of-fit testing on non-Euclidean spaces. Major work has been devoted to the development of goodness-of-fit tests on the circle and sphere (see \cite{rayleigh1880xii,ajne1968simple,beran1968testing,watson1961goodness,watson1962goodness,watson1967another,wellner1979permutation}). The literature on goodness-of-fit testing for the orthogonal group has been limited to three dimensions; two commonly used tests for three dimensional rotations are Downs' generalization 
of the Rayleigh test \citep{downs1972orientation}  and
Prentice's generalization of Gin\'e's $G_n$ test \citep{prentice1978invariant, gine1975invariant}. For a more detailed review of the literature see \cite{mardiadirectional}. In an important development in high dimensional setting, \cite{coram2003new} proposed a family of 
statistical tests for the eigenvalue distribution induced from the Haar measure on the unitary 
group, $U(n)$. Their tests are relatively easy to compute and consistent against all fixed alternatives. One of the new tests in this paper was inspired by the tests of \cite{coram2003new}.

This paper settles the question about the mixing-time of the new sampler using statistical tests. Various known tests are applied (see sections \ref{Sec:RokhGen} and \ref{Sec:KnownTest}), confirming that the new sampler mixes quickly. New tests are introduced (see sections \ref{Sec:TestsBasedOnEigenValues} and \ref{Sec:TestBeyondEigenValues}) and validated using the benchmark examples of section \ref{sec:Benchmark}. The new tests are applied to the new sampler and the results are compared to other tests in section \ref{Sec:Numeric}. The results are in agreement with the claim that the new sampler mixes quickly, i.e.\ after a given number of steps. Local properties, including local power, of the new tests are studied in section \ref{Sec:LocalAsymp}. Similar tests are stated for the other compact groups in section \ref{Sec:OtherGroups}. A further test based on the properties of trace is presented in section 6 of the supplementary material \cite{sepehri2017supplement}.

\subsection{Pseudorandom Orthogonal Transformations.}\label{Sec:RokhGen}
In their recent work, \cite{jones2011randomized} proposed a pseudorandom orthogonal matrix generator which consecutively applies two dimensional rotations in coordinate planes, preconditioned using a Fourier type matrix. This is formally described below.
Suppose $n, M_1, M_2$ are positive integers. Define a pseudorandom $n$-dimensional orthogonal transformation $\Theta$ as a composition of $M_1+M_2+1$ orthogonal operators 
\begin{align}\label{Alg:RokhGen}
\Theta = \left( \Pi_{i=1}^{M_1} Q_i \cdot P_i \right) \cdot F^n \cdot \left( \Pi_{j=1}^{M_2} Q_j P_j \right).
\end{align}
Each $P_i$ and $P_j$ is a uniformly distributed $n \times n$ permutation matrix, independent of others. That is, each $P_i$ corresponds to a permutation $p_i$ of $\{1,\ldots ,n\}$ and $P_i$ acts on vectors as follows
\begin{align*}
(P_i v)_j = v_{p_i(j)}.
\end{align*}
Each $Q_j$ is defined as 
\begin{align*}
Q_j = Q_{n-1,j}\cdot Q_{n-2,j}\cdots Q_{1,j},
\end{align*}
where $ Q_{l,j}$ is a uniform two dimensional rotation in the plane generated by the $l$-th and $l+1$-th coordinates. That is, $(Q_{l,j} v)_i  = v_i$ for $i \neq l,l+1$ and
\begin{align*}
(Q_{l,j} v)_{l}  &=  \cos \theta_{l,j} v_l + \sin \theta_{l,j} v_{l+1}\\
(Q_{l,j} v)_{l+1}  &=  -\sin \theta_{l,j} v_l + \cos \theta_{l,j} v_{l+1} ,
\end{align*}
where $\theta_{l,j}$ is a uniform number in $[0,2\pi]$. All $Q_i$ and $Q_j$ are independent of each other.
 Lastly, the linear operator $F^n$ is defined as follows. Let $d = \floor*{\frac{n}{2}}$ and $T$ be the following $d \times d$ matrix:
\begin{align*}
T_{k,l} = \frac{1}{\sqrt{d}} \exp\left[ - \frac{2\pi i (k-1)(l-1)}{d} \right].
\end{align*}
Define $Z: \mb{R}^{2d} \rightarrow \mb{C}^{d}$ as 
\begin{align*}
[Z(x)]_l = x_{2l-1}+ i x_{2l}.
\end{align*}
For $n$ even, define $F^n$ as 
\begin{align}\label{Def:Kernel}
F^n = Z^{-1} \cdot T\cdot Z.
\end{align}
If $n$ is odd, $F^n$ fixes the last coordinate of $x$, and $F^{n-1}$ defined in (\ref{Def:Kernel}) is applied to the first $n-1$ coordinates. The cost of applying $\Theta$ to vector $x \in \mb{R}^n$ is of order $O(n(\log n +M_1 +M_2))$,  because the cost of applying the operator $F^n$ is $O(n\log n)$ and each operator $Q_j\cdot P_j$ costs $O(n)$. It is claimed in \cite{jones2011randomized} that if $M_1+M_2 = O(\log n)$, then the distribution of $\Theta$ is close to the uniform distribution on the set of all $n \times n$ orthogonal matrices. This makes the new sampler much faster than the state of the art Subgroup Algorithm of \citet{diaconis1987subgroup}, which is an $O(n^3)$ algorithm for generating uniform $n\times n$ rotation matrices. However, it remain to be investigated whether the distribution of the output is close to the uniform distribution. Throughout the paper, the `mixing time' or the number of steps required for the Jones-Osipov-Rokhlin sampler to mix refers to the quantity $M_1+M_2$.

\subsection{Benchmark Examples}\label{sec:Benchmark}
Two benchmark examples of random walks on $SO(n)$ and their mixing properties are used as sanity check for the tests considered in this paper.
\begin{enumerate}
\item\emph{Kac's random walk.} 
The standard \textit{Kac's random walk} $\{O_k\}$ on $SO(n)$ is the defined as follows:
\begin{align*}
O_{k+1} = R_{i,j}^{(k)}(\theta) O_k,
\end{align*}
where $R_{i,j}^{(k)}(\theta)$ is an elementary rotation with angle $\theta$ in the plane generated by the $i$-th and $j$-th coordinate axes, where $\{i,j\}$ is uniformly chosen among all pairs from $\{1,\ldots,n\}$ and $\theta$ uniformly random in $[0,2\pi)$. This walk was introduced as part of Kac's effort to simplify Boltzmann's proof of the H-theorem \citep{kac1959probability}
and Hastings's simulations of random rotations \citep{hastings1970monte}. Convergence of the Kac's random walk has been studied by various authors in different senses. In the current discussion, the focus is on convergence in Wasserstein distance which metrizes the weak convergence; for a review of the literature see \citet{pak2007convergence,oliveira2009convergence,pillai2016mixing}. The best known bound on the mixing-time in Wasserstein distance is obtained by \citet{oliveira2009convergence}, providing an upper bound of order $n^2\log n$ on the mixing-time which is at most a factor $\log n$ away from optimal. For $n=51$, which is the case studied numerically in this paper, $n^2\log n \approx 10000$, but the constants are not known and the actual mixing time could be much smaller of larger than this value.

\item \emph{Product of random reflections.} 

As described in \citet{diaconis2003patterns}, the following random walk on $O(n)$ arose in a telephone encryption problem. At each step, the current orthogonal matrix is multiplied by a random reflection, a matrix of the form $I-2u^Tu$ for a uniform unit vector $u\in \mb{S}^{n-1}$. 

The mixing-time for this chain has been studied carefully in \citet{diaconis1986products,porod1996cut,rosenthal1994random}, proving that $\frac{1}{2} n \log n + cn$ steps are necessary and sufficient for convergence of the reflection walk to the uniform distribution in total variation distance. In fact, \citet{porod1996cut} gives explicit lower- and upper-bounds for the total variation distance between this Markov chain and the Haar measure as a function of $c$, but the bounds are not tight. For $n=51$, $\frac{1}{2}n \log n \approx 100$.
\end{enumerate}

\section{Some Already Known Tests}\label{Sec:KnownTest}
Two important tests for uniformity on $SO(3)$, the Rayleigh's test and the Gine's test, are reviewed in this section. The application to the examples of the previous section is demonstrated.
\subsection{Rayleigh's test}
Perhaps the first test of uniformity on $SO(3)$ was introduced by \citet{rayleigh1880xii}. Given data $g_1,\ldots,g_N \in SO(3)$ define
\begin{align*}
T_R = 3 N \Tr (\bar{g}^T \bar{g}),
\end{align*}
where
\begin{align*}
\bar{g} = \frac{1}{N} \sum_{i=1}^N g_i.
\end{align*}
The Rayleigh's test for uniformity rejects for large values of $T_R$. This can be directly generalized to the higher dimensional case. For any $n \in \mb{N}$ and $g_1,\ldots,g_N \in SO(n)$ define 
\begin{align*}
T_R = n N \Tr (\bar{g}^T \bar{g}),
\end{align*}
where
\begin{align*}
\bar{g} = \frac{1}{N} \sum_{i=1}^N g_i.
\end{align*}

The Rayleigh's test was applied to the benchmark examples of the previous section under the following setup. 
\begin{setup}\label{setup}
The sample size is $N=200$ and the dimension is $n=51$. Each test statistic is computed on 1000 independent repetitions. This setup will be used throughout the paper. All the histograms in this paper are illustrated in blue under the null and in red under the alternatives.
\end{setup}
Figure \ref{Fig:RaylieghRef} illustrates the histograms of the Rayleigh's statistics computed on the product of random reflections and Kac's walk, with that corresponding to the uniform distribution overlaid. Throughout the paper, in all similar figures the color blue corresponds to the Haar distributed samples an d the color colar corresponds to the alternative.

\begin{figure}[h!]
  \centering
  \begin{minipage}[b]{0.24\textwidth}
    \includegraphics[width=\textwidth]{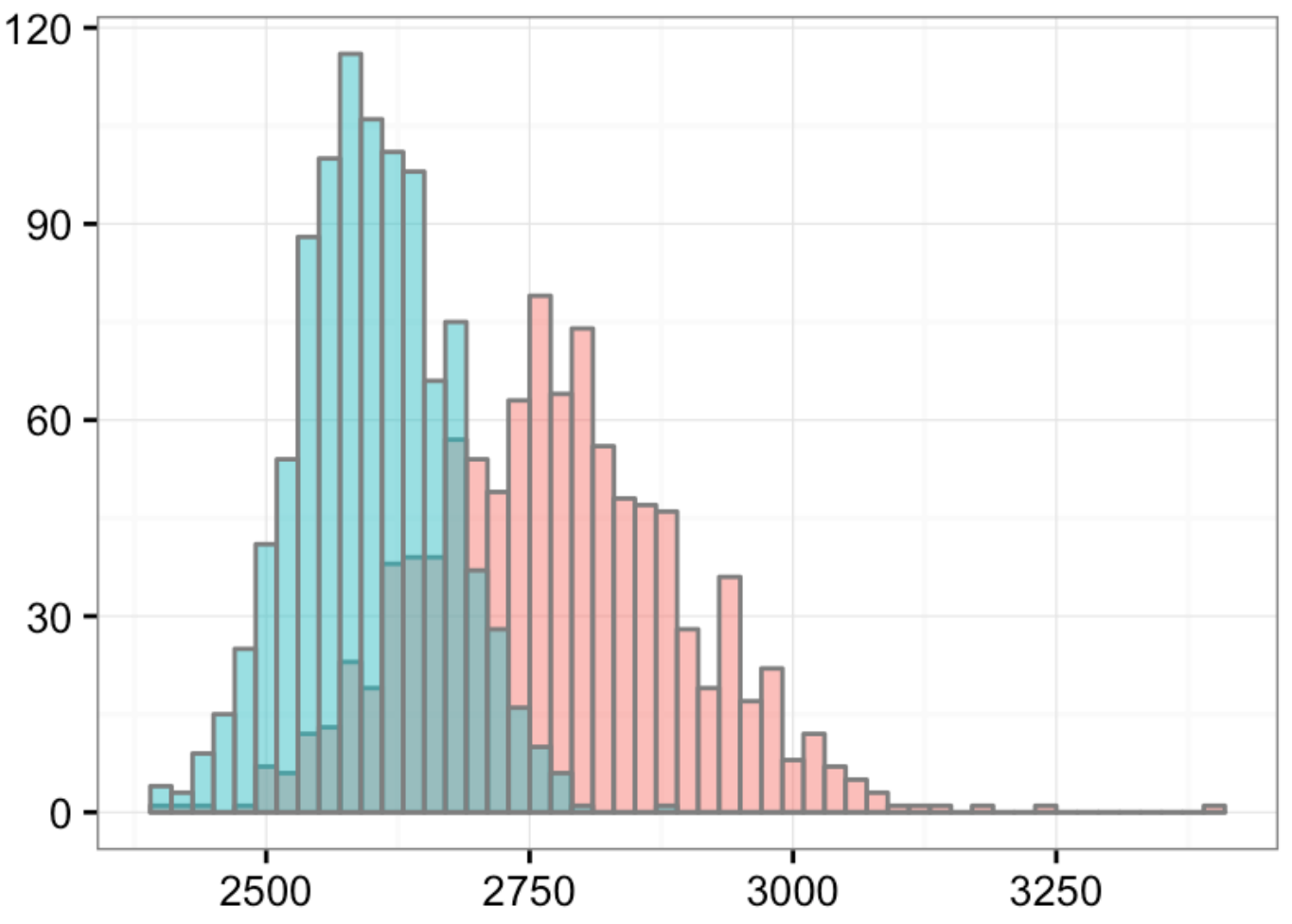}
 %   \caption*{{\scriptsize (a) Uniform sampling}}
  \end{minipage}
  \hfill
  \begin{minipage}[b]{0.24\textwidth}
    \includegraphics[width=\textwidth]{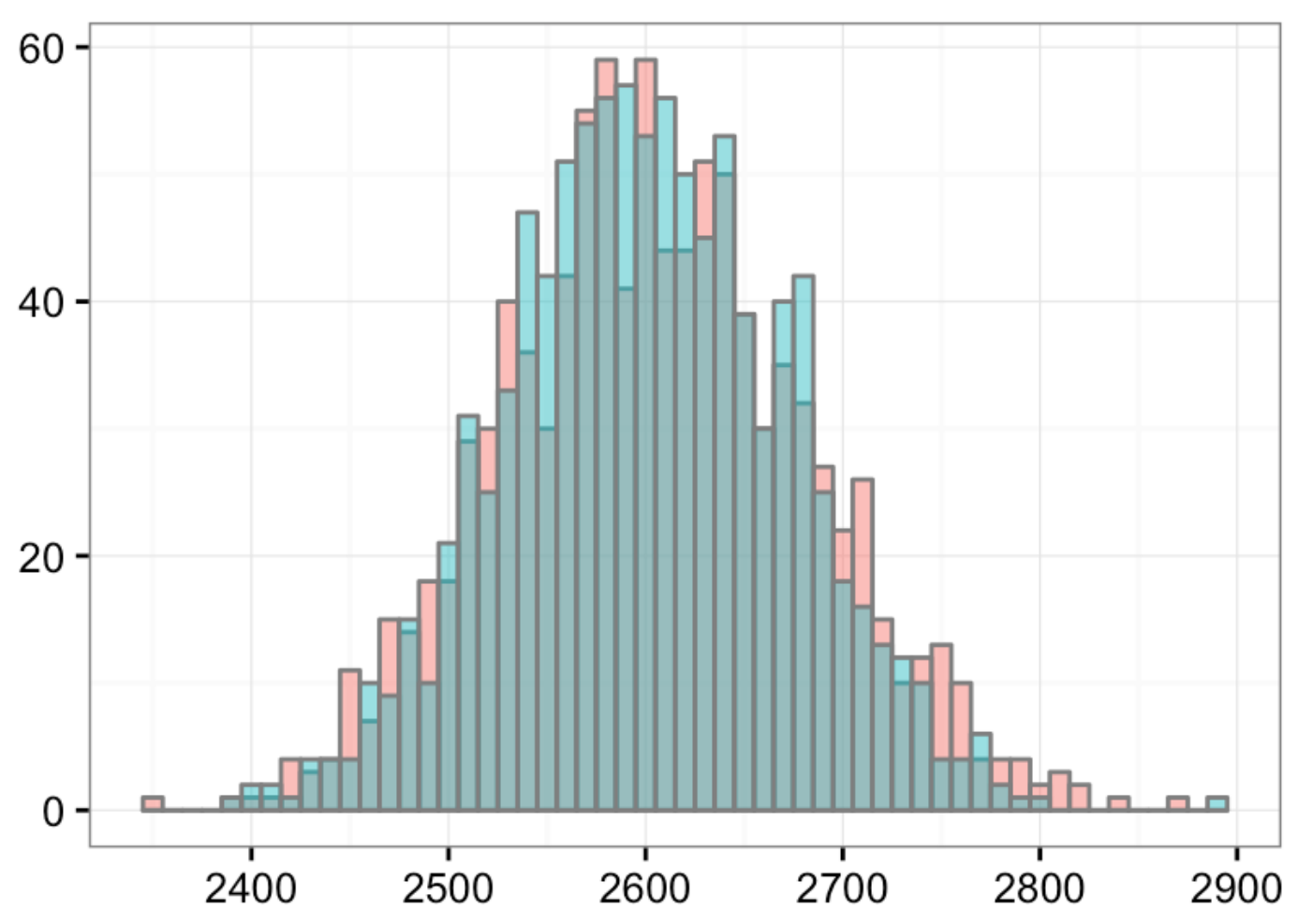}
%    \caption*{{\scriptsize (b) Naive sampling}}
  \end{minipage}\hfill
  \begin{minipage}[b]{0.24\textwidth}
    \includegraphics[width=\textwidth]{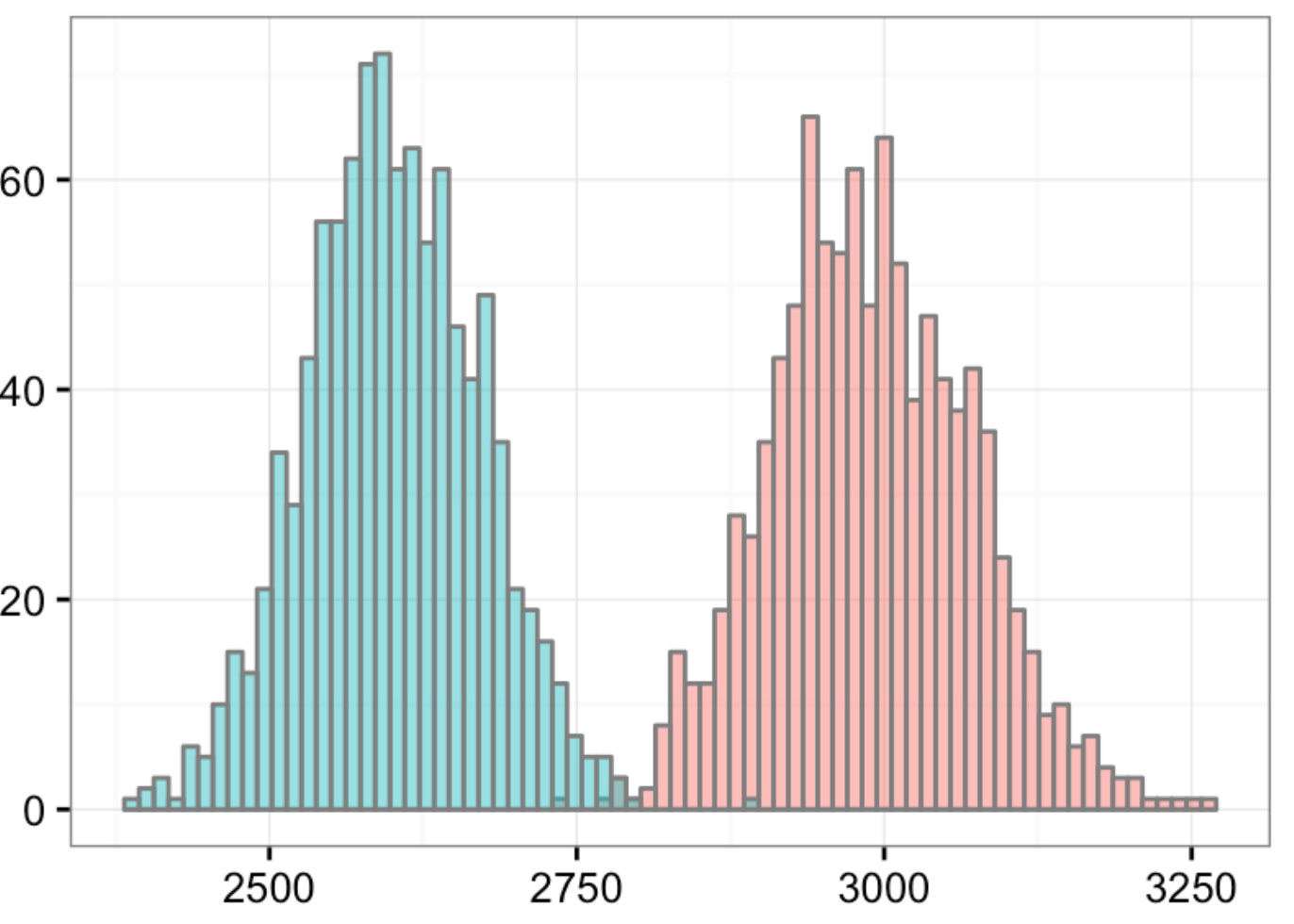}
 %   \caption*{{\scriptsize (a) Uniform sampling}}
  \end{minipage}
  \hfill
  \begin{minipage}[b]{0.24\textwidth}
    \includegraphics[width=\textwidth]{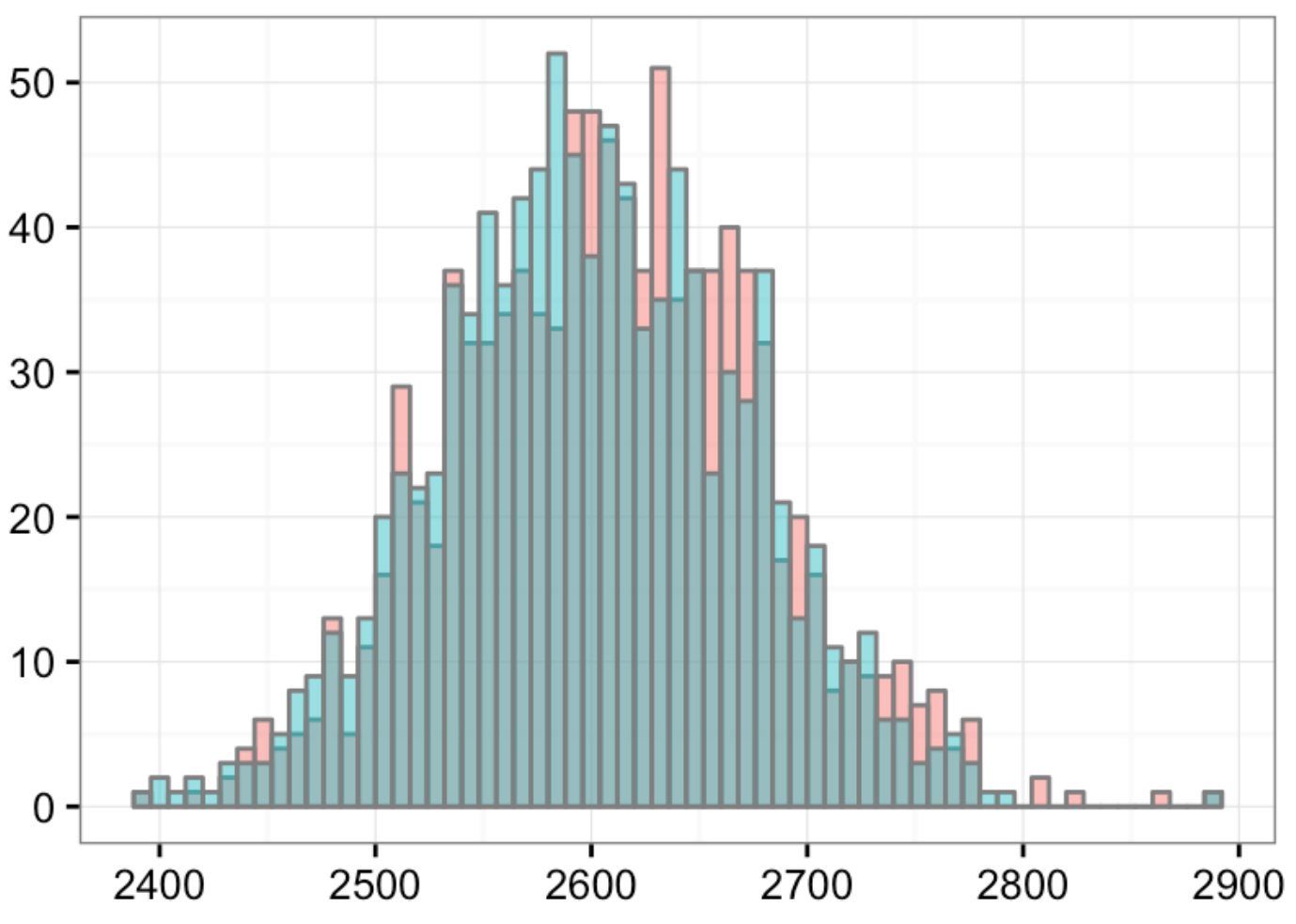}
%    \caption*{{\scriptsize (b) Naive sampling}}
  \end{minipage}
  \caption{Histograms of Rayleigh's statistic under Setup \ref{setup}. The alternatives are (left to right) the Kac's walk after 100 steps, after 150 steps, the product of 90 random reflections, and 140 random reflections.\label{Fig:RaylieghRef}}
\end{figure}

The Rayleigh's test does not seem to have remarkable powerful on either examples. In particular, the Rayleigh's test fails to reject the null hypothesis even after 150 steps of the Kac's walk. The Anderson-Darling p-value for the 1000 values of the Rayleigh's statistic after different number of steps are given in Tables \ref{table:RayleighRef} and \ref{table:RayleighKac} for the product of random reflections and the Kac's walk.\\
\begin{table}[h!]
\caption{$p$-values corresponding to Rayleigh's test on iterated random reflections}
\label{table:RayleighRef}
\centering
\resizebox{\columnwidth}{!}{%
\begin{tabular}{lcccccccccc}
\hline
$\#$ of steps & \multicolumn{1}{c}{\begin{tabular}{@{}c@{}}50\end{tabular}}& \multicolumn{1}{c}{\begin{tabular}{@{}c@{}}75\end{tabular}} & \multicolumn{1}{c}{\begin{tabular}{@{}c@{}}90\end{tabular}} & \multicolumn{1}{c}{\begin{tabular}{@{}c@{}}100\end{tabular}}&\multicolumn{1}{c}{\begin{tabular}{@{}c@{}}110\end{tabular}} &\multicolumn{1}{c}{\begin{tabular}{@{}c@{}}125  \end{tabular}} &\multicolumn{1}{c}{\begin{tabular}{@{}c@{}}140\end{tabular}}& \multicolumn{1}{c}{\begin{tabular}{@{}c@{}}150\end{tabular}}& \multicolumn{1}{c}{\begin{tabular}{@{}c@{}}175\end{tabular}}& \multicolumn{1}{c}{\begin{tabular}{@{}c@{}}200\end{tabular}}  \\
\hline
 A-D test &$\ll$1e-32 & $\ll$1e-32 & $\ll$1e-32 & $\ll$1e-32 & $\ll$1e-32 & 4.1e-16 & 0.03 & 0.37 & 0.84 & 0.39\\
 \hline
\end{tabular}
}
\end{table}
\begin{table}[h!]
\caption{$p$-values corresponding to Rayleigh's test on Kac's walk}
\label{table:RayleighKac}
\centering
\resizebox{\columnwidth}{!}{%
\begin{tabular}{lccccccccc}
\hline
$\#$ of steps & \multicolumn{1}{c}{\begin{tabular}{@{}c@{}}100\end{tabular}}& \multicolumn{1}{c}{\begin{tabular}{@{}c@{}}150\end{tabular}} & \multicolumn{1}{c}{\begin{tabular}{@{}c@{}}200\end{tabular}} & \multicolumn{1}{c}{\begin{tabular}{@{}c@{}}250\end{tabular}}&\multicolumn{1}{c}{\begin{tabular}{@{}c@{}}300\end{tabular}} &\multicolumn{1}{c}{\begin{tabular}{@{}c@{}}350  \end{tabular}} &\multicolumn{1}{c}{\begin{tabular}{@{}c@{}}400\end{tabular}}& \multicolumn{1}{c}{\begin{tabular}{@{}c@{}}450\end{tabular}}& \multicolumn{1}{c}{\begin{tabular}{@{}c@{}}500\end{tabular}}\\
\hline
 A-D test &$\ll$1e-32 & 0.08 & 0.58 & 0.23 & 0.23 & 0.86 & 0.70 & 0.82 & 0.89\\
 \hline
\end{tabular}
}
\end{table}

Applying Rayleigh's test to samples generated by the new sampler of \citet{jones2011randomized} provides no evidence for departure from uniformity, after only one iteration. The p-value of the Anderson-Darling test under Setup \ref{setup} is $0.35$.

\subsection{Gine's test}
Another important test of uniformity on $SO(3)$ was introduced by \citet{gine1975invariant}. Given data $g_1,\ldots,g_N \in SO(3)$, define
\begin{align*}
T_G = \frac{1}{N} \sum_{i=1}^N \sum_{j=1}^N \sqrt{\Tr (I - g_i^Tg_j)}.
\end{align*}
Gine's test rejects for large values of $T_G$. It is consistent against all fixed alternatives on $SO(3)$, but not in any higher dimensions. The corresponding test was carried out on the benchmark examples and the new sampler. Gine's tests seems to be more powerful than the Rayleigh's test on these examples. Histograms of the values of the Gine's statistic are illustrated in Figure \ref{Fig:GineRef}.
\begin{figure}[h!]
  \centering
  \begin{minipage}[b]{0.24\textwidth}
    \includegraphics[width=\textwidth]{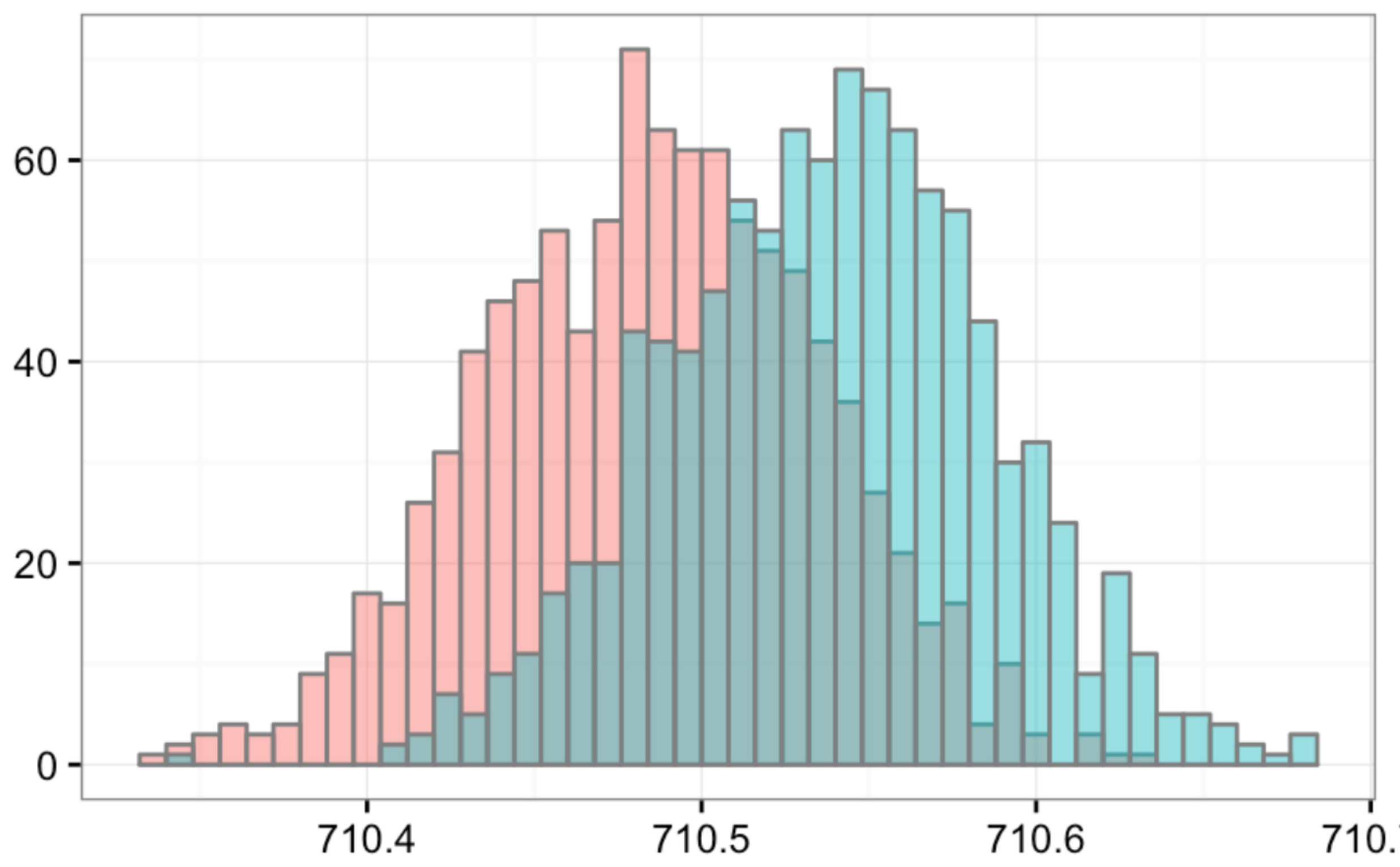}
 %   \caption*{{\scriptsize (a) Uniform sampling}}
  \end{minipage}
  \hfill
  \begin{minipage}[b]{0.24\textwidth}
    \includegraphics[width=\textwidth]{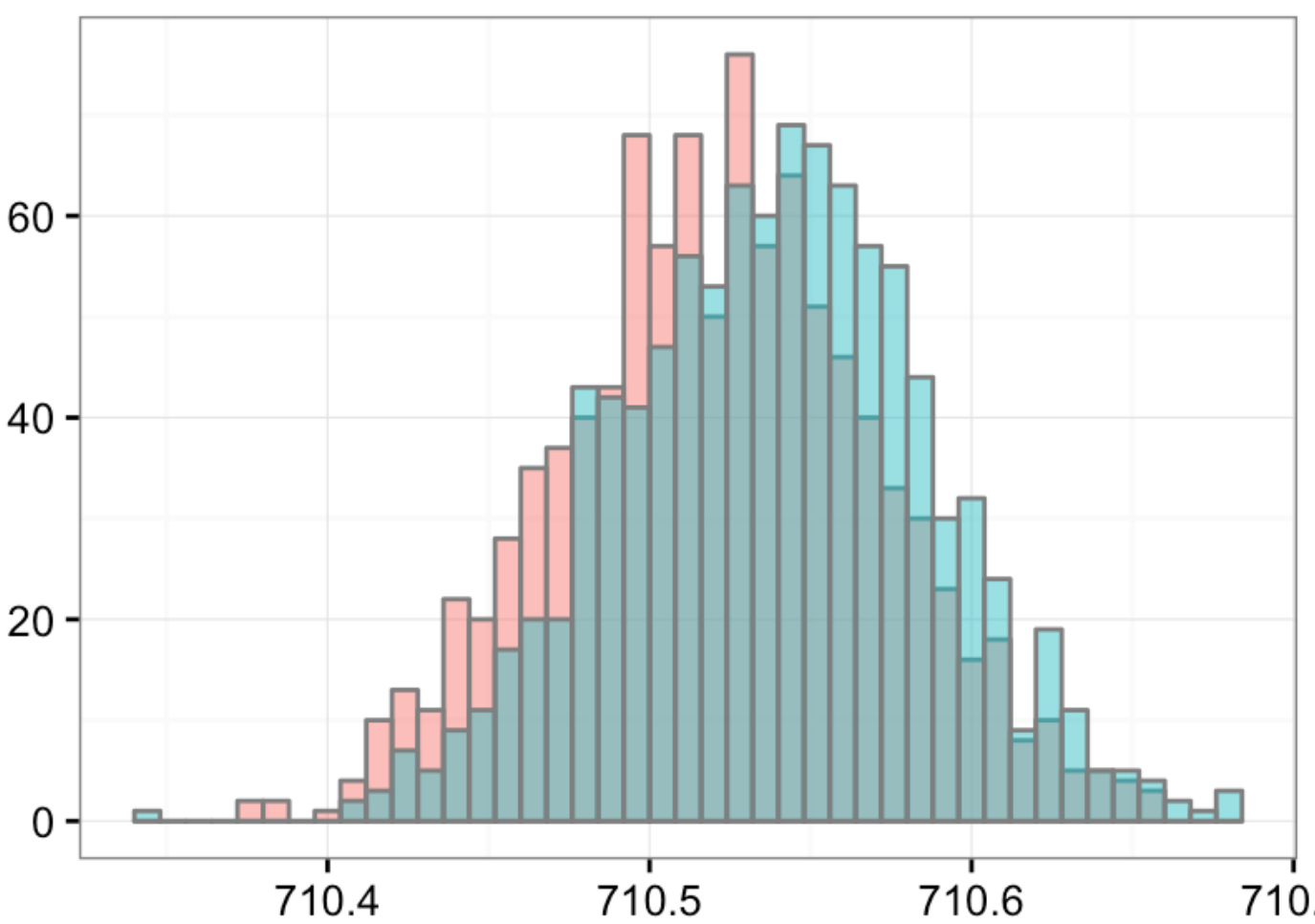}
%    \caption*{{\scriptsize (b) Naive sampling}}
  \end{minipage}\hfill
   \begin{minipage}[b]{0.24\textwidth}
    \includegraphics[width=\textwidth]{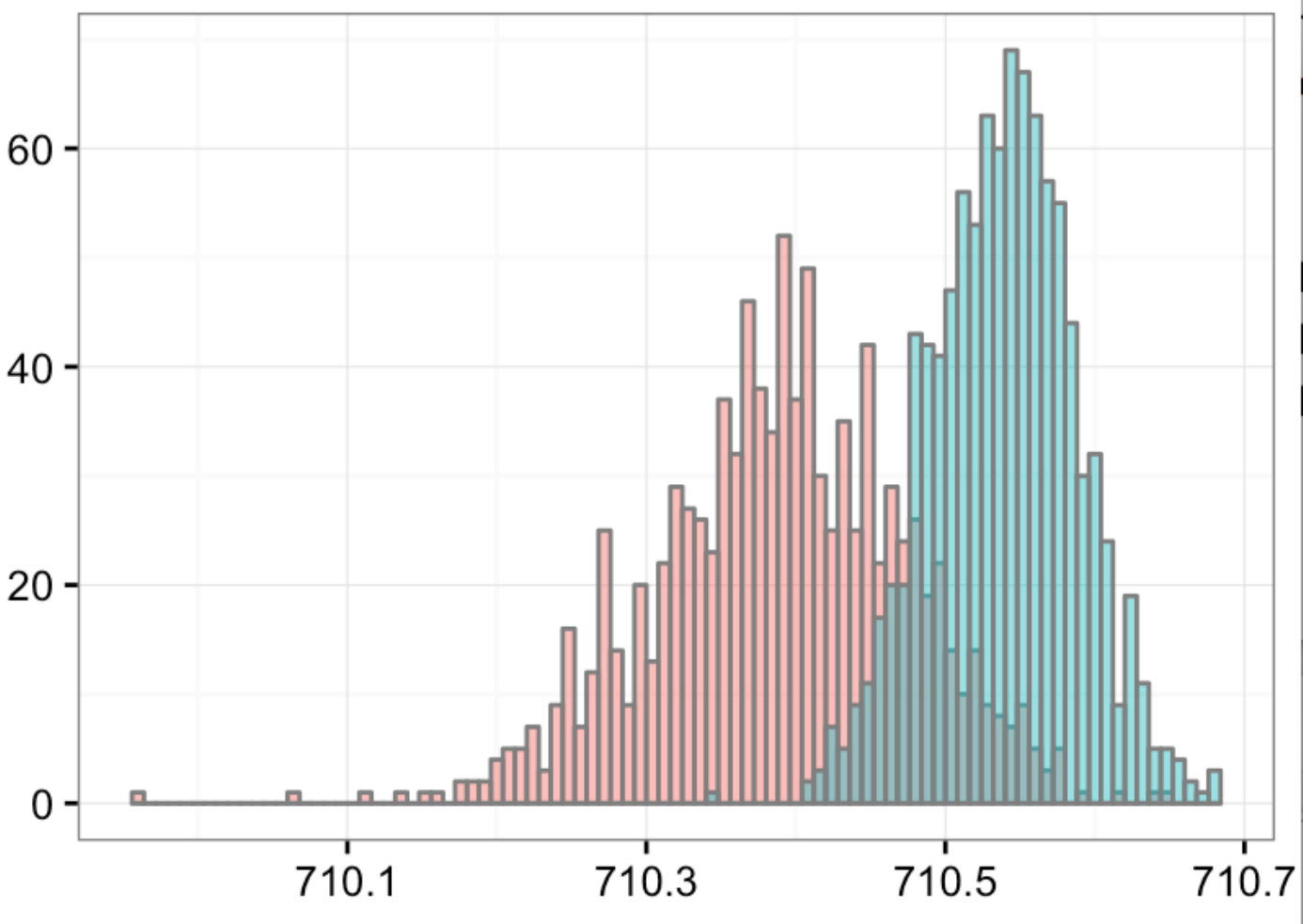}
 %   \caption*{{\scriptsize (a) Uniform sampling}}
  \end{minipage}
  \hfill
  \begin{minipage}[b]{0.24\textwidth}
    \includegraphics[width=\textwidth]{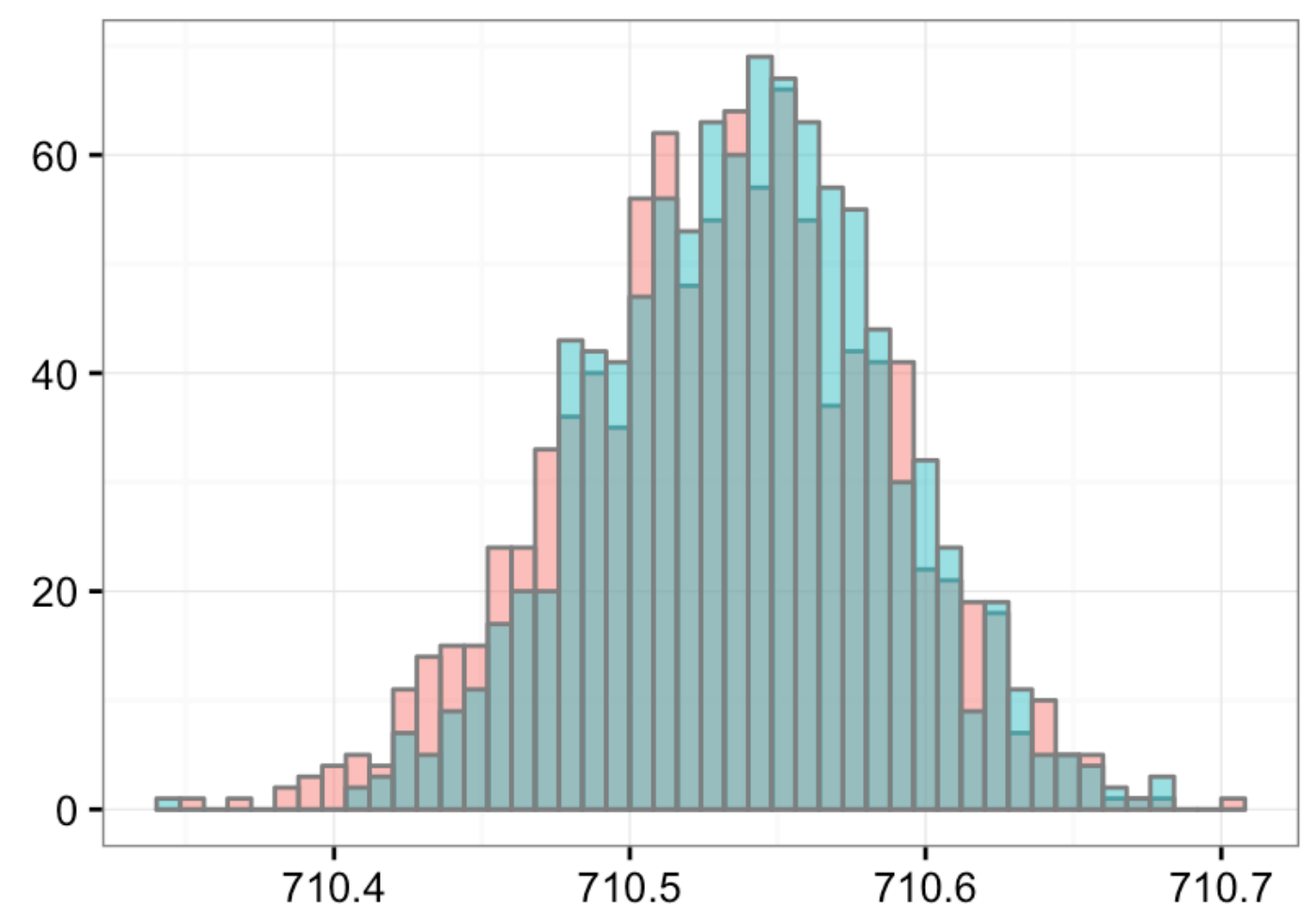}
%    \caption*{{\scriptsize (b) Naive sampling}}
  \end{minipage}
  \caption{Histogram of Gine's statistic under Setup \ref{setup}. The alternatives are (left to right) product of 110 random reflections, 125 random reflections, the Kac's walk after 100 steps, and 150 steps.\label{Fig:GineRef}}
\end{figure}

Applying the Gine's test to the new sampler provides no evidence for departure from the null. The Anderson-Darling p-values after one and two iterations of the sampler are $0.58$ and $0.45$, respectively.

\section{Tests Based on Eigenvalues}\label{Sec:TestsBasedOnEigenValues}
The new sampler has passed all the tests considered in the previous sections. In this section, various new tests based on the eigenvalues are introduced and applied to the benchmark examples as well as the new sampler.

\subsection{A test based on exponential families}
The joint density of the eigenvalues of a uniformly random $g \in SO(2n+1)$ is given by Weyl \cite[page 224]{weyl1946classical} as
\begin{align}\label{JontDensityEigenvalues}
f(e^{\pm i\theta_1},\ldots , e^{\pm i\theta_n}) \propto \prod_i \sin^2 (\frac{\theta_i}{2}) \prod_{i<j} |\cos \theta_i - \cos \theta_j|^2,
\end{align}
where $(1,e^{\pm i\theta_1},\ldots , e^{\pm i\theta_n})$ are eigenvalues of $g$. By a change of variables $x_i = \cos \theta_i$, the density, in terms of $(x_1,\ldots ,x_n)\in [-1,1]^n$, becomes
\begin{align}\label{Eq:Density}
f(x_1, \ldots , x_n) \propto \prod_{i<j} |x_i - x_j|^2 \prod_i \frac{\sqrt{1-x_i}}{\sqrt{1+x_i}}.
\end{align}
The density $f$ can be embedded in the following exponential family:
\begin{align}\label{Def:ExpFam}
f_{\gamma, \alpha, \beta} (x_1, \ldots , x_n) \propto \prod_{i<j} |x_i - x_j|^{2\gamma} \prod_i (1-x_i)^{\alpha-1}(1+x_i)^{\beta -1}.
\end{align}
The normalizing constant is given by Selberg's integral \cite[pg. 320, eqn. (17.5.9)]{mehta2004random} as
\begin{align*}
&\int_{[-1,1]^n} \prod_{i<j} |x_i - x_j|^{2\gamma} \prod_i (1-x_i)^{\alpha-1}(1+x_i)^{\beta -1} dx_1 \ldots dx_n \\&=
2^{\gamma n(n-1)+ n (\alpha + \beta -1)} \prod_{j=0}^{n-1} \frac{\Gamma(1+\gamma + j \gamma)\Gamma (\alpha + j \gamma) \Gamma (\beta + j \gamma)}{\Gamma(1+\gamma) \Gamma (\alpha + \beta + \gamma (n+j -1))}.
\end{align*}
The density (\ref{Eq:Density}) is the special case of (\ref{Def:ExpFam}) for $\gamma_0 =1$ , $\alpha_0 = 3/2$, and $\beta_0 = 1/2$. We abuse the notation to denote by $f_{1,\frac{3}{2},\frac{1}{2}}$ both densities (\ref{JontDensityEigenvalues}) and (\ref{Eq:Density}).
Recall that the eigenvalues of a uniform orthogonal matrix $g$ are placed quite regularly on the unit circle. For example, the trace of $g$ is approximately normal with mean zero and variance one \citep{diaconis1990trace}; whereas, for uniformly distributed points on the unit circle, say in conjugate pairs, because of the Law of Large Number and the Central Limit Theorem, the sum has mean zero and variance of order $O(n)$. In particular, the magnitude of the sum is of order $O(\sqrt{n})$. The family of densities $f_{\gamma, \alpha, \beta}$ models the regularity of the configuration of the points on the circle. The case of $\gamma = 0, \alpha = \beta = 1/2$ of (\ref{Def:ExpFam}) corresponds to $\theta_i$ being independent uniform on $[0,\pi]$. As $\gamma$ tends to infinity the points $e^{i\theta_j}$ become evenly placed on the semicircle. This is illustrated in Figure \ref{Fig:EigSpacing}.

\begin{figure}[h!]
  \centering
  \begin{minipage}[b]{0.32\textwidth}
    \includegraphics[width=\textwidth]{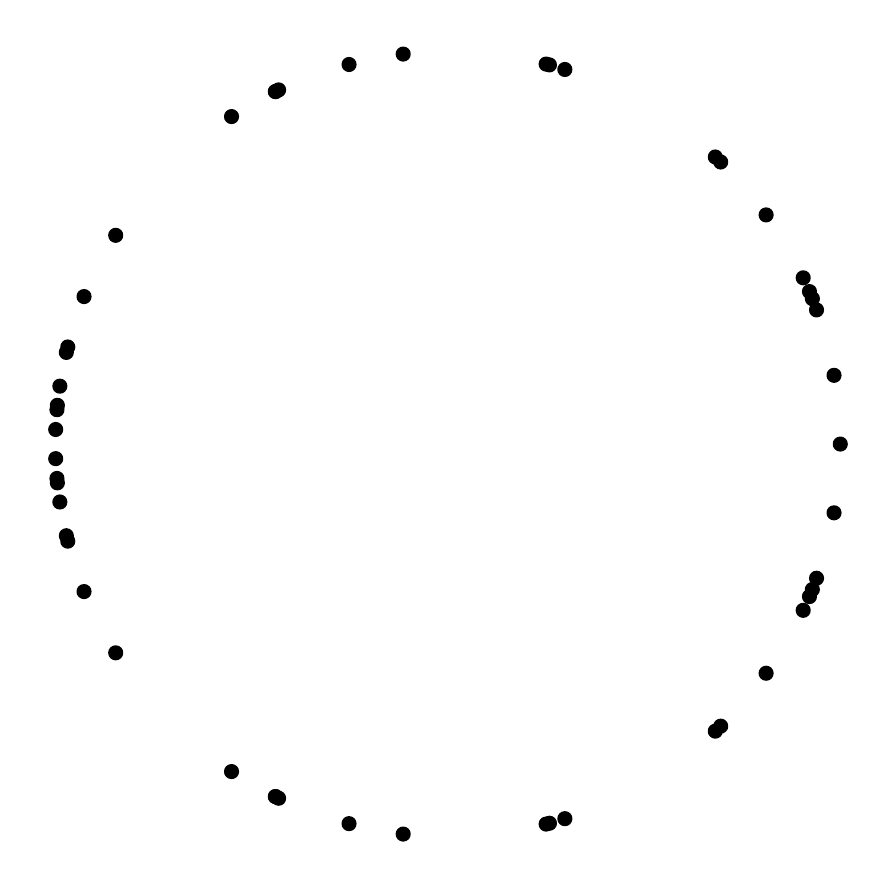}
 %   \caption*{{\scriptsize (a) Uniform sampling}}
  \end{minipage}
  \hfill
  \begin{minipage}[b]{0.32\textwidth}
    \includegraphics[width=\textwidth]{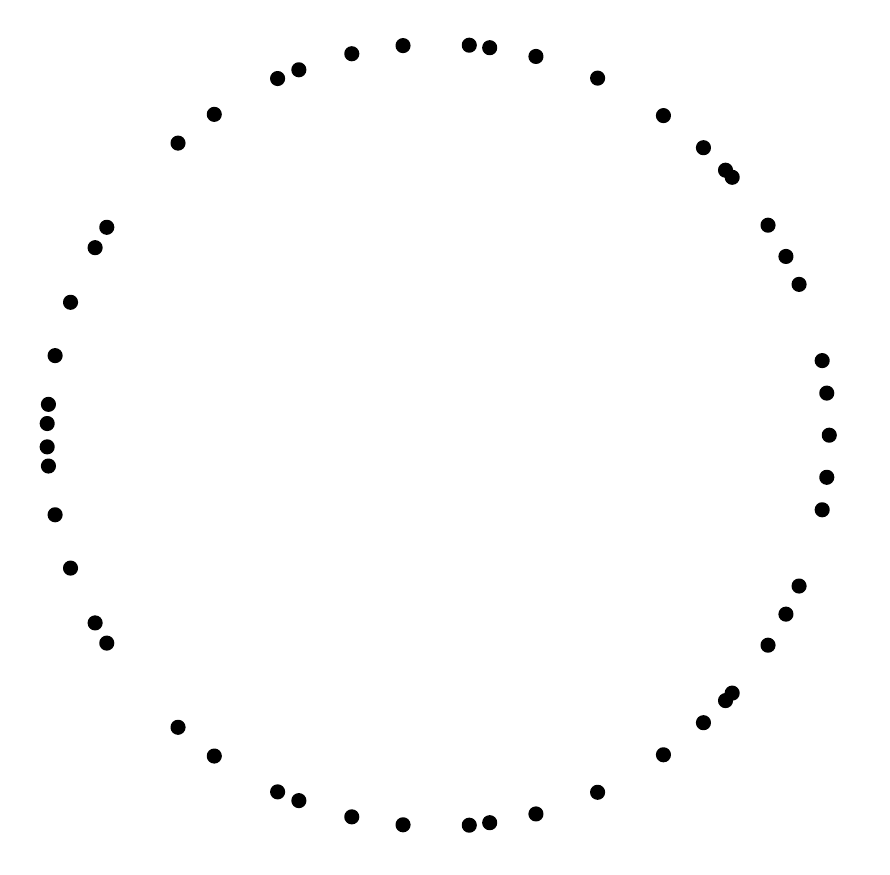}
%    \caption*{{\scriptsize (b) Naive sampling}}
  \end{minipage}
  \hfill
   \begin{minipage}[b]{0.32\textwidth}
    \includegraphics[width=\textwidth]{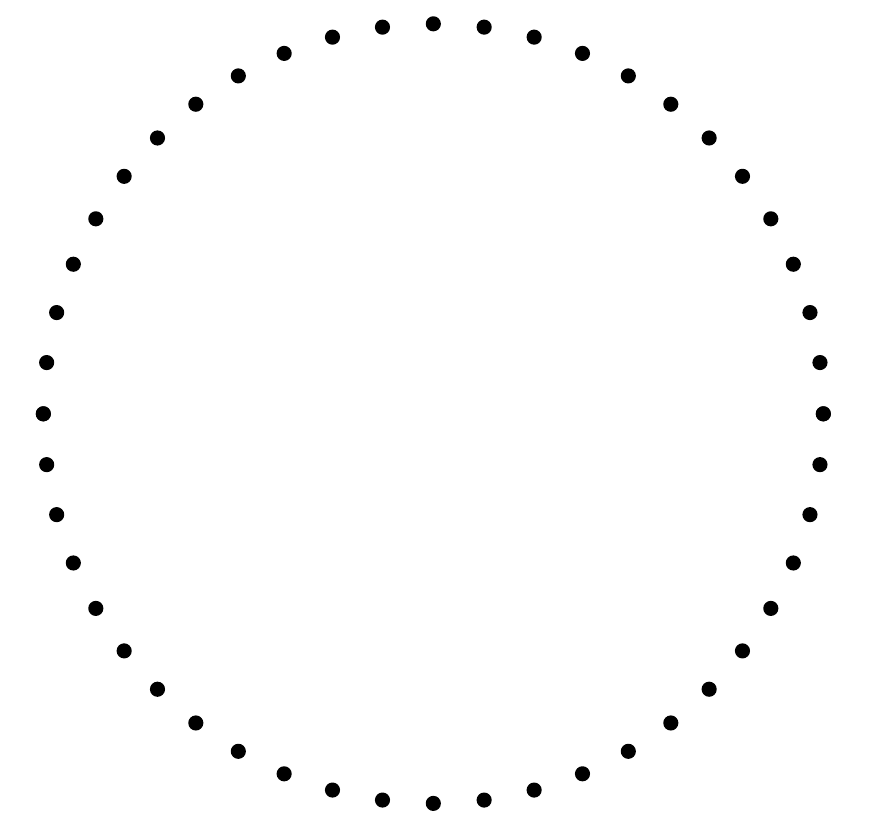}
   % \caption*{{\scriptsize (a) Uniform sampling}}
  \end{minipage}
  \caption{left: 25 uniform points on the circle and their conjugates, middle: eigenvalues of a uniform $51\times 51$ orthogonal matrix, right: 51 evenly placed points on the circle.\label{Fig:EigSpacing}. The sum of all complex numbers is approximately: 7.8 on the left, 0.78 in the middle, and 0 on the right.}
\end{figure}

Testing for the Haar measure in this parametric family translates to
\begin{align}
H_0: (\gamma, \alpha, \beta) = (\gamma_0, \alpha_0, \beta_0) \quad \text{vs} \quad H_1: (\gamma, \alpha, \beta) \neq (\gamma_0, \alpha_0, \beta_0) .
\end{align}
Unfortunately, there is no uniformly most powerful test available in this setting. However, the problem fits into the framework of Asymptotically Normal Experiments of Le Cam. This allows for construction of an asymptotically \textit{maximin} optimal test as follows.
Define $T_1(x) = 2\sum_{i<j} \log |x_i-x_j|$, $T_2(x) = \sum_i \log(1-x_i)$, and $T_3(x) = \sum_i \log (1+x_i)$. Then $(T_1,T_2,T_3)$ is a sufficient statistic for the exponential family $f_{\gamma, \alpha, \beta}$. That is,
\begin{align*}
f_{\gamma, \alpha, \beta}(x) = \exp\left( \gamma T_1(x)+ (\alpha-1)T_2(x)+(\beta-1)T_3(x)-\mc{A}(\gamma, \alpha, \beta)\right),
\end{align*}
where $\mc{A}(\gamma, \alpha, \beta) = \log \left( 2^{\gamma n(n-1)+ n (\alpha + \beta -1)} \prod_{j=0}^{n-1} \frac{\Gamma(1+\gamma + j \gamma)\Gamma (\alpha + j \gamma) \Gamma (\beta + j \gamma)}{\Gamma(1+\gamma) \Gamma (\alpha + \beta + \gamma (n+j -1))}\right)$.
Then, given data $x^{(1)},\ldots,x^{(N)}$ the likelihood is
\begin{align*}
\prod_m f_{\gamma, \alpha, \beta}(x^{(m)}) = \exp\left( \gamma T_1^{(N)}+ (\alpha-1)T_2^{(N)}+(\beta-1)T_3^{(N)}-N\mc{A}(\gamma, \alpha, \beta)\right),
\end{align*}
where $T_i^{(N)} = \sum_m T_i(x^{(m)})$. By standard asymptotic theory, under the null, $T^{(N)} = \frac{1}{N} (T_1^{(N)},T_2^{(N)},T_3^{(N)})$ is approximately normal with mean $\mu = \nabla \mc{A}(\gamma_0, \alpha_0, \beta_0)$ and covariance matrix $\Sigma = \nabla^2 \mc{A}(\gamma_0, \alpha_0, \beta_0)/N$. The test that rejects for large values of
\begin{align*}
T = (T^{(N)}-\mu)^T \Sigma^{-1} (T^{(N)}-\mu)
\end{align*}
is asymptotically maximin optimal \cite[Theorem 13.5.5]{lehmann2006testing}. Moreover, $\mu$ and $\Sigma$ can be computed using the recurrence relations and series expansions for digamma and trigamma functions \cite[pp 258-259 ]{abramowitz1964handbook}. For $51\times 51$ orthogonal matrices, $n= 25$,  and $(\gamma_0,\alpha_0,\beta_0) = (1,3/2,1/2)$ they are
\begin{align*}
\mu \approx (-329.70,  -14.73 , -19.92),
\end{align*}
\[
\Sigma \approx
  \begin{bmatrix}
     5.67712\times 10^{-2}& -1.40363\times 10^{-2} & -3.61067\times 10^{-4}\\
    -1.40363 \times 10^{-2}&  1.52093\times 10^{-2} & -3.46523\times 10^{-3}\\
   -3.61067 \times 10^{-4}& -3.46523 \times 10^{-3}&  3.96834 \times 10^{-2}
  \end{bmatrix}.
\]

We applied this test on two samples of size $N = 200$ and $1000$; data was generated using only one step of the pseudorandom sampler (\ref{Alg:RokhGen}) of section \ref{Sec:RokhGen}. The test statistic $T$ evaluated to $3.84$ and $1.34$, respectively. Under the null, $T$ is approximately $\chi^2_3$ distributed. The corresponding $p$-values are $0.72$ and $0.28$, respectively; there is no evidence for departure from uniformity.

To further explore the performance of $T$, it was applied to the benchmark examples.
Tables \ref{table:ExpFamRef} and \ref{table:ExpFamKac} show the $p$-values corresponding to $T$ applied to the benchmark examples, suggesting that $T$ detects the cut-off to some extent. In particular, it seems to be more powerful than both Gine's and Rayleigh's tests.

\begin{table}[h!]
\caption{$p$-values corresponding to product of random reflections}
\label{table:ExpFamRef}
\centering
%\resizebox{\columnwidth}{!}{%
\begin{tabular}{lcccccc}
\hline
$k=$ & \multicolumn{1}{c}{\begin{tabular}{@{}c@{}}100\end{tabular}} & \multicolumn{1}{c}{\begin{tabular}{@{}c@{}}125\end{tabular}} & \multicolumn{1}{c}{\begin{tabular}{@{}c@{}}130\end{tabular}}&\multicolumn{1}{c}{\begin{tabular}{@{}c@{}}135\end{tabular}} &\multicolumn{1}{c}{\begin{tabular}{@{}c@{}}140  \end{tabular}} &\multicolumn{1}{c}{\begin{tabular}{@{}c@{}}150\end{tabular}}  \\
\hline
 $N=200$ & 0 & 0.0002 &0.001& 0.055& 0.09& 0.20\\
 $N=1000$& 0 & 0& 0 &0.0001 &0.03&0.42\\
 \hline
\end{tabular}
%}
\end{table}

\begin{table}[h!]
\caption{$p$-values corresponding to Kac's walk}
\label{table:ExpFamKac}
\centering
%\resizebox{\columnwidth}{!}{%
\begin{tabular}{lcccccc}
\hline
$k=$ & \multicolumn{1}{c}{\begin{tabular}{@{}c@{}}150\end{tabular}} & \multicolumn{1}{c}{\begin{tabular}{@{}c@{}}200\end{tabular}} & \multicolumn{1}{c}{\begin{tabular}{@{}c@{}}225\end{tabular}}&\multicolumn{1}{c}{\begin{tabular}{@{}c@{}}240\end{tabular}} &\multicolumn{1}{c}{\begin{tabular}{@{}c@{}}250  \end{tabular}} &\multicolumn{1}{c}{\begin{tabular}{@{}c@{}}300\end{tabular}}  \\
\hline
  $N=200$ & 0 & 0 & 0.001 & 0.004 &0.19& 0.40\\
 $N=1000$& 0 & 0& 0 & 0&0.00001 & 0.65\\
 \hline
\end{tabular}
%}
\end{table}

\subsection{A family of consistent tests for $f_{1,\frac{3}{2},\frac{1}{2}}$} \label{section:EigenvalueTest}
Given data $x^{(1)},\ldots,x^{(N)}$ in $[-1,1]^n$, this section introduces a family of tests $T_z^{(N)}(x^{(1)},\ldots,x^{(N)})$, for a parameter $0< z <1$, that are invariant under the natural symmetries of $f_{1,\frac{3}{2},\frac{1}{2}}$ and are consistent against all alternatives. The asymptotic distributions under null and alternative are also available. The construction is given for a general hypothesis testing problem in the following section; It is then carried out for $f_{1,\frac{3}{2},\frac{1}{2}}$.
\subsubsection{Spectral tests on general spaces}\label{Sec:SpectralTestsGeneralSpaces}
Let $\mc{X}$ be a Polish space and $\mu$ a probability measure 
 on $\mc{X}$. Consider the standard non-parametric goodness-of-fit
testing problem: given independent and identically distributed observations
$x_1,\ldots,x_N \in \mc{X}$ from a probability measure $\nu$ on $\mc{X}$,
test if $\nu = \mu$. A general test based on spectral techniques can be constructed as follows. 

Let $\mc{L}^2(\mc{X},\mu)$ be the space of square $\mu$-integrable 
functions on $\mc{X}$. Assume that $\mc{L}^2(\mc{X},\mu)$ is separable with 
a countable orthonormal basis $\{f_i\mid i\ge 0\}$, with $f_0 = 1$. For example, it suffices to assume that $\mc{X}$ is compact. Theorem 13.8 in \citet{thomson2008elementary} gives a condition for $L^2$ to be separable in the general setting. Define the empirical measure of $\{x_i\}$ 
as $\nu_N = \frac{1}{N}\sum_{i=1}^N \delta_{x_i}$, and the Fourier coefficients of $\nu_N$ as 
\begin{align*}
\widehat{\nu}_N(i) &= \int_{\mc{X}}f_i(x)\nu_N(dx)\\
				 &= \frac{1}{N}\sum_{n=1}^N f_i(x_n).
\end{align*}
Then, under the null, $\widehat{\nu}_N(i) \rightarrow 0$ as $n$ grows to infinity, for $i>0$. This property characterizes $\mu$; that is, if $x_1,x_2,\ldots \in \mc{X}$ are i.i.d.\ draws from $\nu$ and $\widehat{\nu}_N(i) \rightarrow 0$ for all $i>0$ then, $\mu = \nu$. This property can be used to construct tests of fit for $\mu$. By the central limit theorem
\begin{align*}
\sqrt{N}\widehat{\nu}_N(i) \rightarrow \mc{N}(0,1).
\end{align*}
Thus, $N |\widehat{\nu}_N(i)|^2$ is asymptotically $\chi^2_1$ distributed. For a sequence of weights $\bm{c} = (c_1,c_2,\ldots)$, define
\begin{align*}
T_{\bm{c}} = N \sum_i c_i |\widehat{\nu}_N(i)|^2.
\end{align*}
Assuming that $T_{\bm{c}}$ converges to a finite value, a test that rejects for large values of $T_{\bm{c}}$ can be used for testing $H_0: \nu = \mu$. Many well-known classical tests can be constructed in this manner. The most important example is the celebrated Neyman's smooth test for uniformity on the unit interval. Neyman's test uses Legendre polynomials as the orthonormal basis. Under mild conditions and the assumption $c_i >0 $ the test based on $T_{\bm{c}}$ is consistent against all alternatives, and has various desired statistical properties. There is a vast literature on properties of tests of this form; we do not attempt to review the literature since it is considered classical nowadays.

There are two main challenges in using the above machinery in a general problem: 1) finding an orthonormal basis for $\mc{L}^2$, 2) computing $T_{\bm{c}}$. In his celebrated paper, \cite{gine1975invariant} gave a solution for the first challenge for the testing problem with $\mu$ being the invariant measure on a compact Riemannian manifold $M$; this is sketched below.

Let $\Delta$ be the Laplace-Beltrami operator 
 (Laplacian) of $M$ acting on the space of Schwartz functions
  by duality. Denote by $E_k$ the $k$-th invariant eigenspace
   of $\Delta$ with eigenvalue $\sigma_k$. Let $\{f_i^k\}_{i=0}^{\dim E_k}$
    be an orthonormal basis for $E_k$. Then, $\{f_i^k\mid   k \ge 0 ,\; 1 \le i\le \dim E_k\}$ is an orthonormal basis for $\mc{L}^2(M,\mu)$. Note that the hypothesis testing problem
    \begin{align*}
    H_0: \nu = \mu \quad \text{vs} \quad H_1: \nu \neq \mu
    \end{align*}
is invariant under natural symmetries of $M$. Therefore, by the Hunt-Stein theorem \cite[page 331]{lehmann2006testing} one only needs to consider invariant tests. \cite{gine1975invariant} suggested the test, called Sobolev test, based on
\begin{align*}
T_N^{\bm{\alpha}}(\bm{x}) = N \sum_{k=1}^\infty \alpha_k \sum_{f_i\in E_k} \left[\int_M f_i\, d\nu_N(\bm{x})\right]^2,
\end{align*}
for a sequence of weights $\bm{\alpha} = (\alpha_1,\alpha_2,\ldots)$ such that $\sup |\alpha_k \sigma_k^s|<\infty$ for some
      $s > \frac{1}{2}\dim M$. Note that the weights depend only on the eigenspaces; this ensures that the test is invariant. \cite{gine1975invariant} studied statistical properties of
 the Sobolev tests; he derived the null and alternative distribution and investigated
  local optimality properties following \cite{beran1968testing}.
  
  Although these tests have been successful in practice, usually substantial 
non-trivial work is required to carry out the details for any particular example of interest. \cite{gine1975invariant} carried out the program for the circle, sphere, and the projective
  plane, recovering many known examples in the literature and introducing new tests of uniformity.
Several authors have studied and derived Sobolev tests for different examples 
including circular and directional data, tests of symmetry, and unitary eigenvalues; see \cite{prentice1978invariant,wellner1979permutation,jupp1983sobolev,jupp1985sobolev, hermans1985new,baringhaus1991testing,sengupta2001optimal,coram2003new}.

Regarding the second challenge, note that
\begin{align*}
T_{\bm{c}} &=N \sum_k c_k |\widehat{\nu}_N(k)|^2  \\
		&= \frac{1}{N} \sum_{i,j=1}^N \sum_{k}c_k  f_k(x_i)f_k(x_j).
\end{align*}
Therefore, it suffices to have a way of computing
\begin{align}\label{Def:KernelGeneral}
K(x,y) = \sum_{k} c_k f_k(x)f_k(y).
\end{align}

 To compute $K(\cdot,\cdot)$, \cite{gine1975addition,gine1975invariant} suggested partial answers for the Sobolev tests, based on Zonal functions; there still remains the challenge to find a closed form for $T_N^{\bm{\alpha}}$, or to compute it effectively for a general Reimannian manifold $M$. The paper resolves this issue for compact classical groups and a class of weight sequences $\{ c_k\}$.

\subsubsection{Spectral tests for $f_{1,3/2,1/2}$}\label{Sec:EigenTest}
The following facts and notation will be used throughout this paper. For each partition $\lambda$, with at most $n$ parts, of an arbitrary non-negative integer, there exists an integer $d_\lambda$ and a map $\pi_\lambda$ from $SO(2n+1)$ to the set of all $d_\lambda \times d_\lambda$ matrices with the following properties. For $g,h \in SO(2n+1)$, one has $\pi(g\cdot h) = \pi(g) \pi(h)$; the set of all matrix coordinates $\{\pi_\lambda^{i,j}\mid \lambda , 1\le i,j \le d_\lambda\}$ is an orthonormal basis for $\mc{L}^2(SO(2n+1))$. Moreover, if $\chi_\lambda (g) = \Tr (\pi_\lambda (g))$, then $\{\chi_\lambda\}$ is an orthonormal basis for $\mc{L}^2(f_{1,3/2,1/2})$. These are standard facts from representation theory of Lie groups. A brief introduction is given in section 1 of the supplementary material \cite{sepehri2017supplement}. For a textbook treatment see \cite{bump2004lie,goodman2009symmetry}.

Given independent observations $g_1,\ldots,g_N \in SO(2n+1)$, define the Fourier coefficient corresponding to $\lambda$ as
\begin{align*}
\widehat{\chi}_N(\lambda)={1\over N} \sum_{i=1}^N \chi_\lambda(g_i).
\end{align*} 
For $0<z<1$ define the test statistics $T_z^{(N)}$ as 
\begin{align}\label{Def:EigenTest}
T_z^{(N)} &= N \sum_{\lambda\neq 0} z^{|\lambda|} \left| \widehat{\chi}_N(\lambda) \right|^2,
\end{align}
where $|\lambda| = \lambda_1 + \ldots + \lambda_n$ is the sum of the parts of the partition $\lambda$ and sum is over all partitions of all positive integers with at most $n$ parts.

To use $T_z^{(N)}$ in practice, a closed form expression for the kernel given in (\ref{Def:KernelGeneral}),
\begin{align*}
K_z(g_i,g_j) = \sum_{\lambda\neq 0} z^{|\lambda|}  \chi_\lambda(g_i)\chi_\lambda(g_j),
\end{align*}
would yield a closed form expression for $T_z^{(N)}$.

 \citet{coram2003new} used the closed form expression for $K_z(g,h)$, given by the Cauchy identity for the Schur functions, to build a test for the eigenvalue distribution induced by the Haar measure on the unitary group. The test statistic $T_z^{(N)}$ is an analogue for the orthogonal group of their test. 

In the case of the orthogonal group there was no closed form for $K_z$ available in the literature. Motivated by the testing problem under study in the present paper, the author derived Cauchy identities for all of the compact classical groups.
\begin{prop}[Cauchy identity for $SO(2n+1)$, Theorem 2.1 in \cite{sepehri2017supplement}] \label{thm:CauchySOodd}
Let $g,h \in SO(2n+1)$ have eigenvalues equal to $(1, e^{\pm i \theta_1},\ldots,e^{\pm i \theta_n})$ and $(1, e^{\pm i \phi_1},\ldots,e^{\pm i \phi_n})$, respectively. Then, 
\begin{align}\label{Eq:CauchyIdentity}
K_z(g,h) &= \frac{(1-z)^n \det\left[ \frac{(1+z)^2+2z(\cos\theta_k+\cos\phi_l)}
{(1+z^2)^2 - 4(z+z^3)\cos \theta_k \cos \phi_l+2z^2 (\cos 2\theta_k +\cos 2\phi_l)}
\right]_{k,l}}{(4z)^{\binom{n}{2}}\prod_{i<j} \left(\cos\theta_i-\cos\theta_j\right) \prod_{i<j} 
\left(\cos\phi_i-\cos\phi_j\right)}-1.
\end{align}
\end{prop}
Despite the complicated appearance of the formula (\ref{Eq:CauchyIdentity}), it is relatively easy to compute if the dimension is not too large, offering a way to compute $T_z^{(N)}$.

The test based on $T_z^{(N)}$ was applied to the benchmark examples and the new sampler. It is indeed more powerful than all tests considered in the previous sections on both examples. Figure \ref{fig:HisEigenTestRef} illustrates the histogram of the values of $T_z^{(N)}$ under Setup \ref{setup}. The $5\%$-level test based on $T_{1/2}^{(200)}$ has power equal to $0.64$ against the product of 140 random reflections; the power drops to 0.30 after 150 steps. Similarly, it has power equal to 0.93 against the Kac's walk after 250 steps, which drops to 0.25 after 300 steps. 
\begin{figure}[h!]
  \centering
  \begin{minipage}[b]{0.24\textwidth}
    \includegraphics[width=\textwidth]{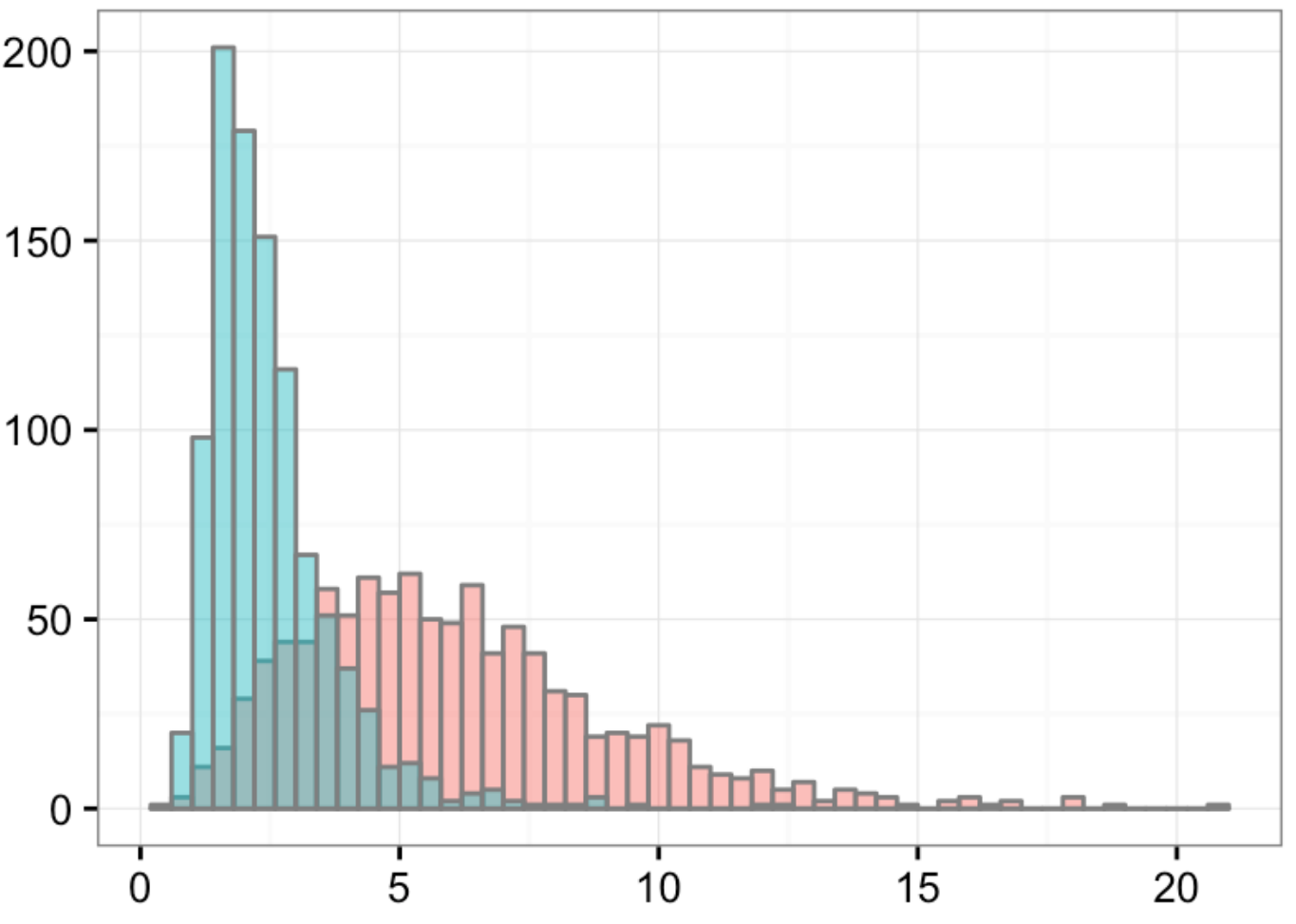}
 %   \caption*{{\scriptsize (a) Uniform sampling}}
  \end{minipage}
  \hfill
  \begin{minipage}[b]{0.24\textwidth}
    \includegraphics[width=\textwidth]{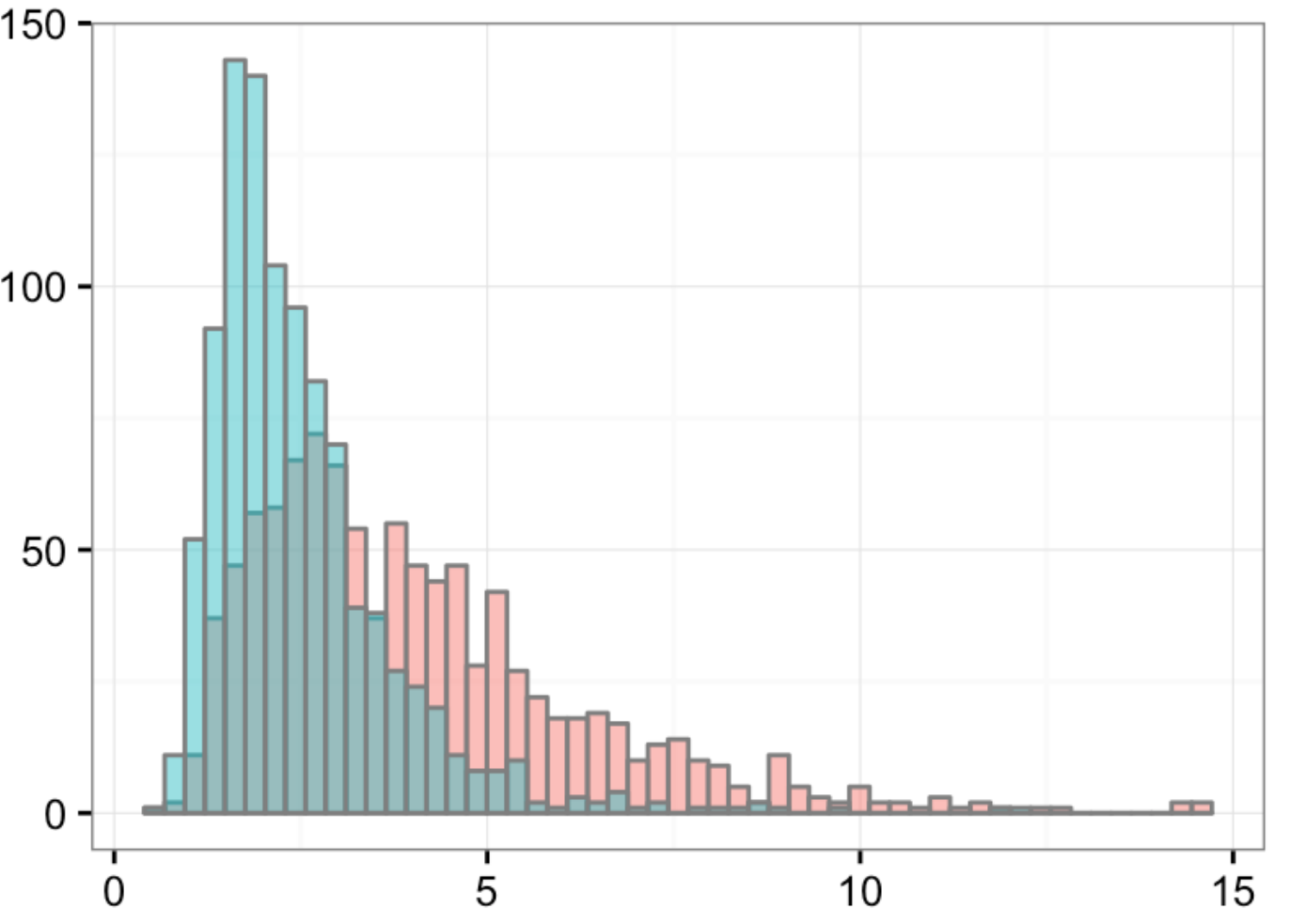}
%    \caption*{{\scriptsize (b) Naive sampling}}
  \end{minipage}\hfill
  \begin{minipage}[b]{0.24\textwidth}
    \includegraphics[width=\textwidth]{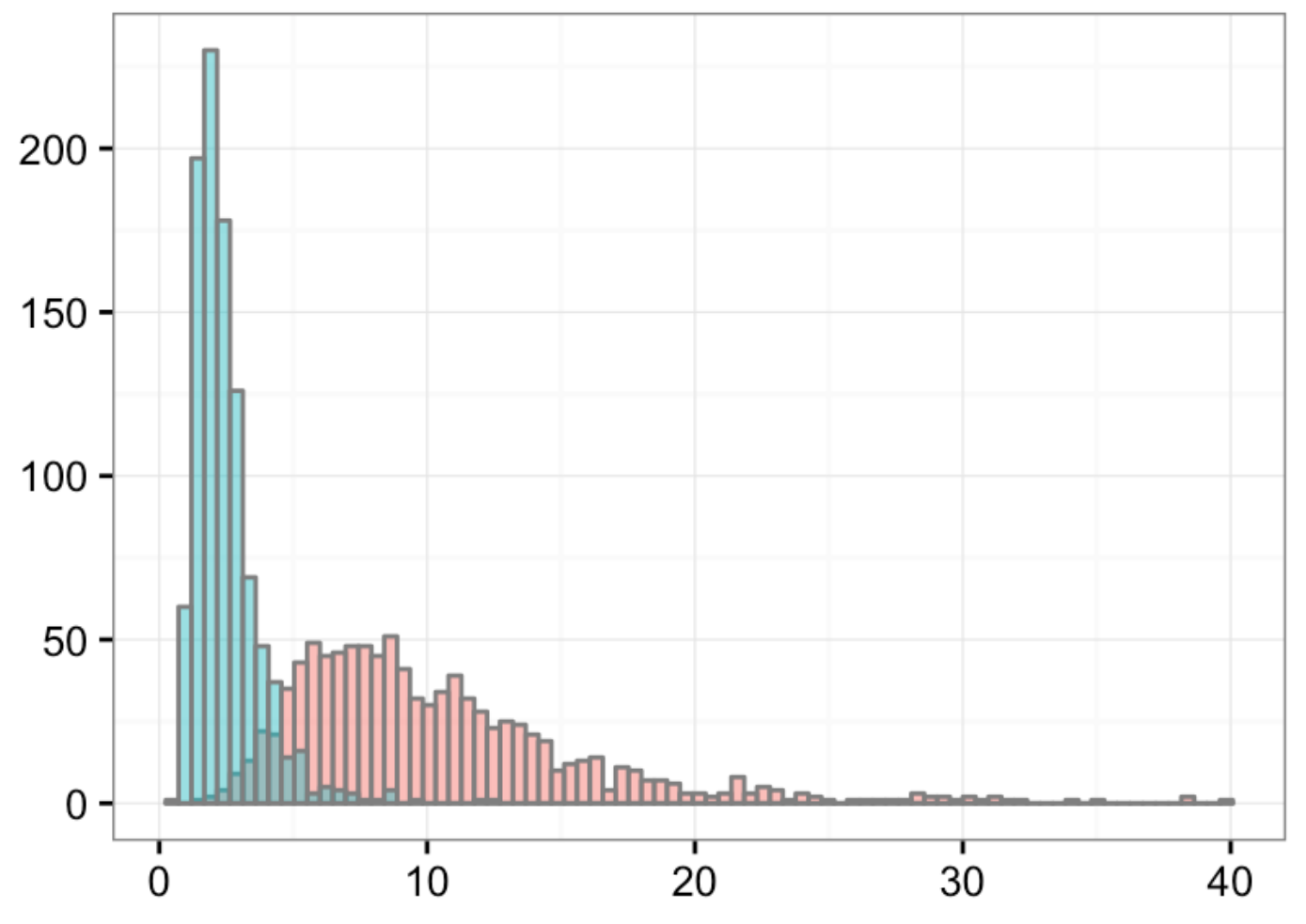}
 %   \caption*{{\scriptsize (a) Uniform sampling}}
  \end{minipage}
  \hfill
  \begin{minipage}[b]{0.24\textwidth}
    \includegraphics[width=\textwidth]{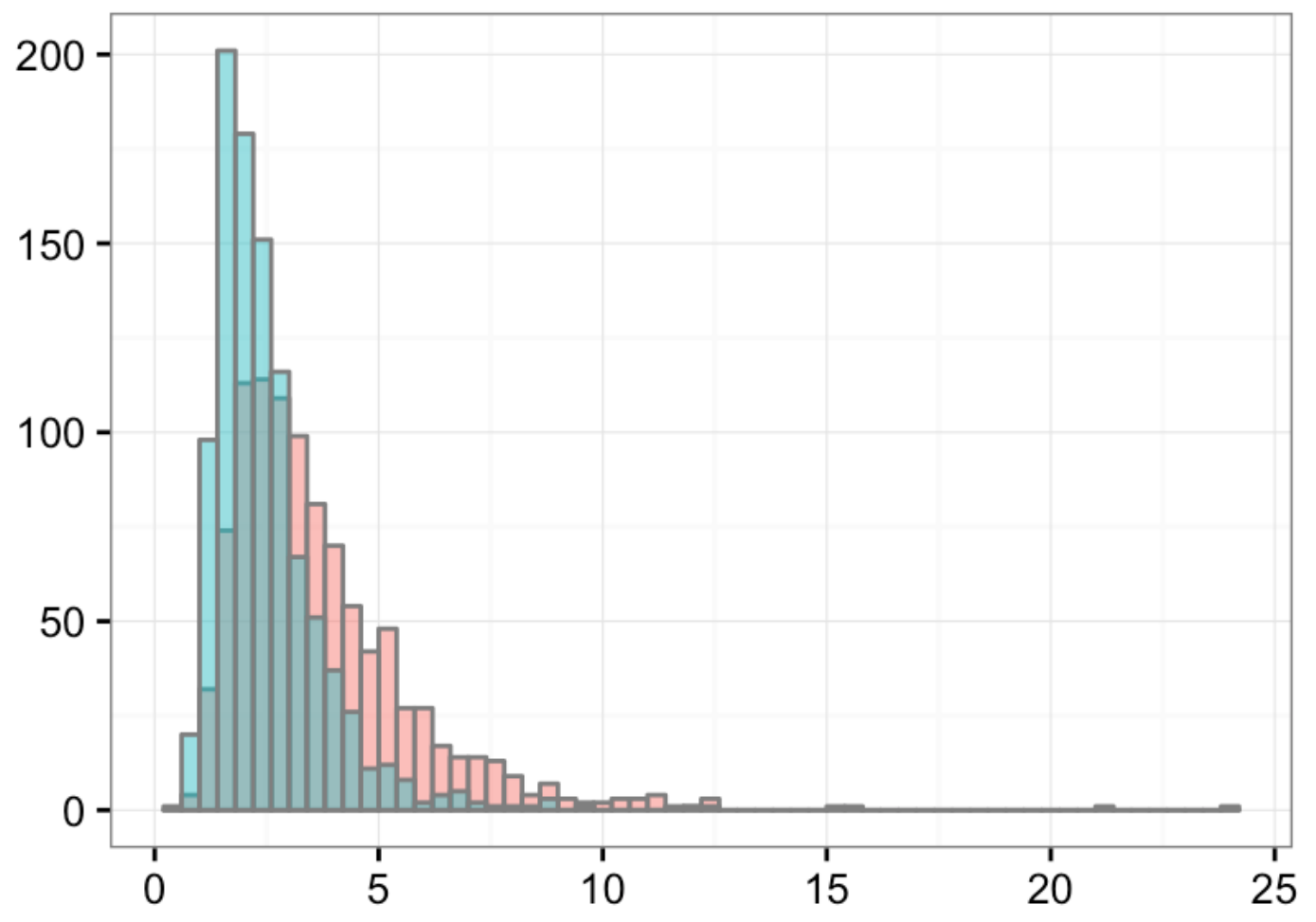}
%    \caption*{{\scriptsize (b) Naive sampling}}
  \end{minipage}
   \caption{Histogram of $T_z^{(N)}$ under Setup \ref{setup} and $ z= 0.5,$. The alternatives are (left to right) the product of 140 random reflections, 150 random reflections, the Kac's random walk after 250 steps, and 300 steps.}
  \label{fig:HisEigenTestRef}
\end{figure}

\begin{remark} The test statistic $T_z^{(N)}$ has, by construction, a decomposition to approximately independent parts. When the test rejects the null hypothesis, it would be illuminating regarding the nature of the departure from uniformity to see which of the components is larger than its typical values. This was investigated using lower order terms of the form
\begin{align*}
C_k^{(N)} = \frac{1}{N} \sum_{i=1}^N \Tr (g_i^k),
\end{align*}
 under Setup \ref{setup} and result is shown in Figures \ref{fig:ADValuesRefComp} and \ref{fig:ADValuesKacComp} for $k=1,\ldots,5$.
 As it can be seen, the component corresponding to $\Tr (g)$ is the most significant against the product of random reflections. For the Kac's walk, $\Tr (g^2)$ is the most significant and captures the cutoff more clearly compared to $T_z^{(N)}$. The nature of deviation differs for these examples.
 \begin{figure}[h!]
  \centering
  \begin{minipage}[b]{0.7\textwidth}
    \includegraphics[width=\textwidth]{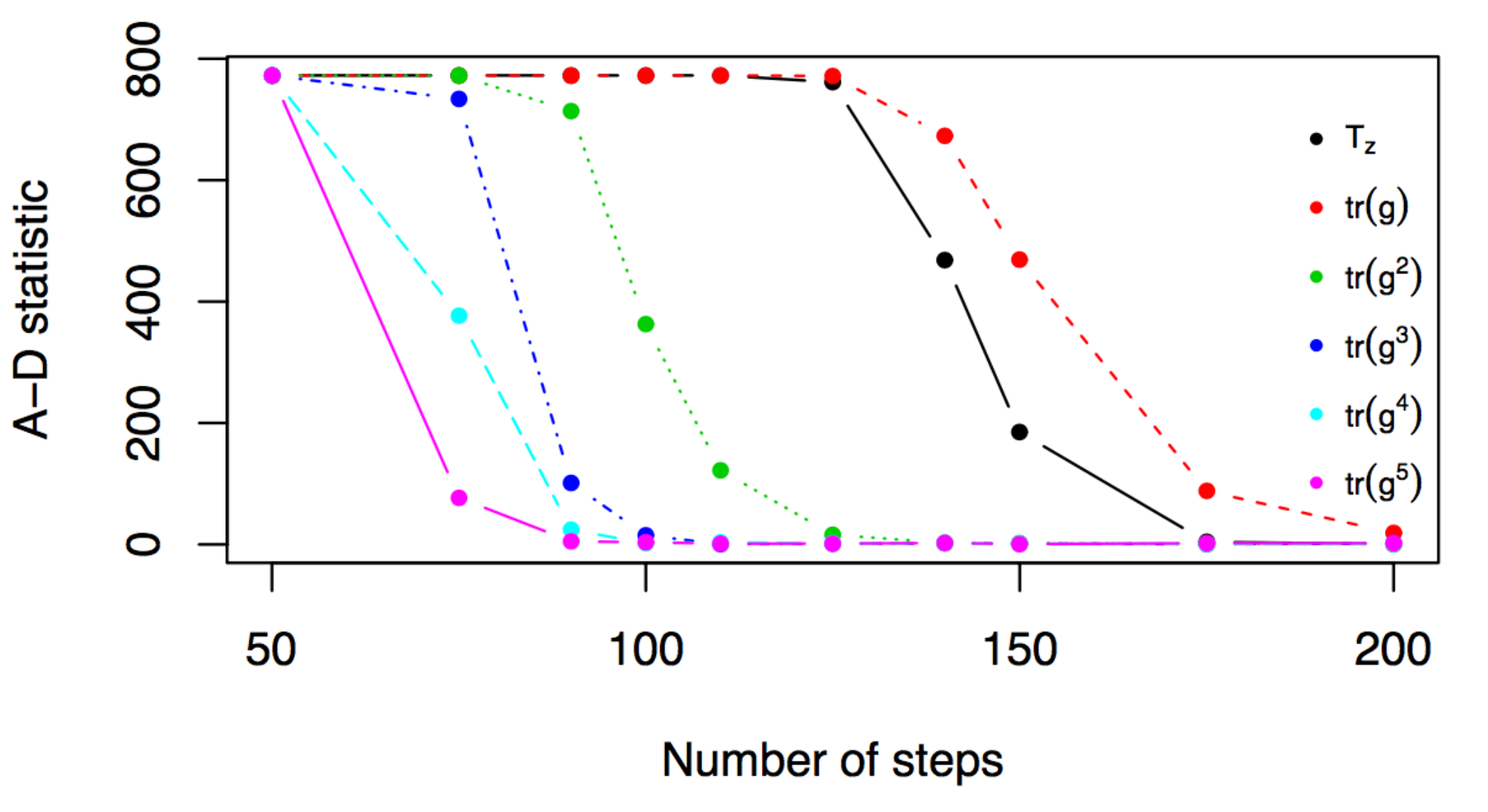}
 %   \caption*{{\scriptsize (a) Uniform sampling}}
  \end{minipage}
   \caption{Values of the Anderson-Darling statistic for comparison of the uniform sample and the product of random reflections for different tests.}
  \label{fig:ADValuesRefComp}
\end{figure} 
\begin{figure}[h!]
  \centering
  \begin{minipage}[b]{0.7\textwidth}
    \includegraphics[width=\textwidth]{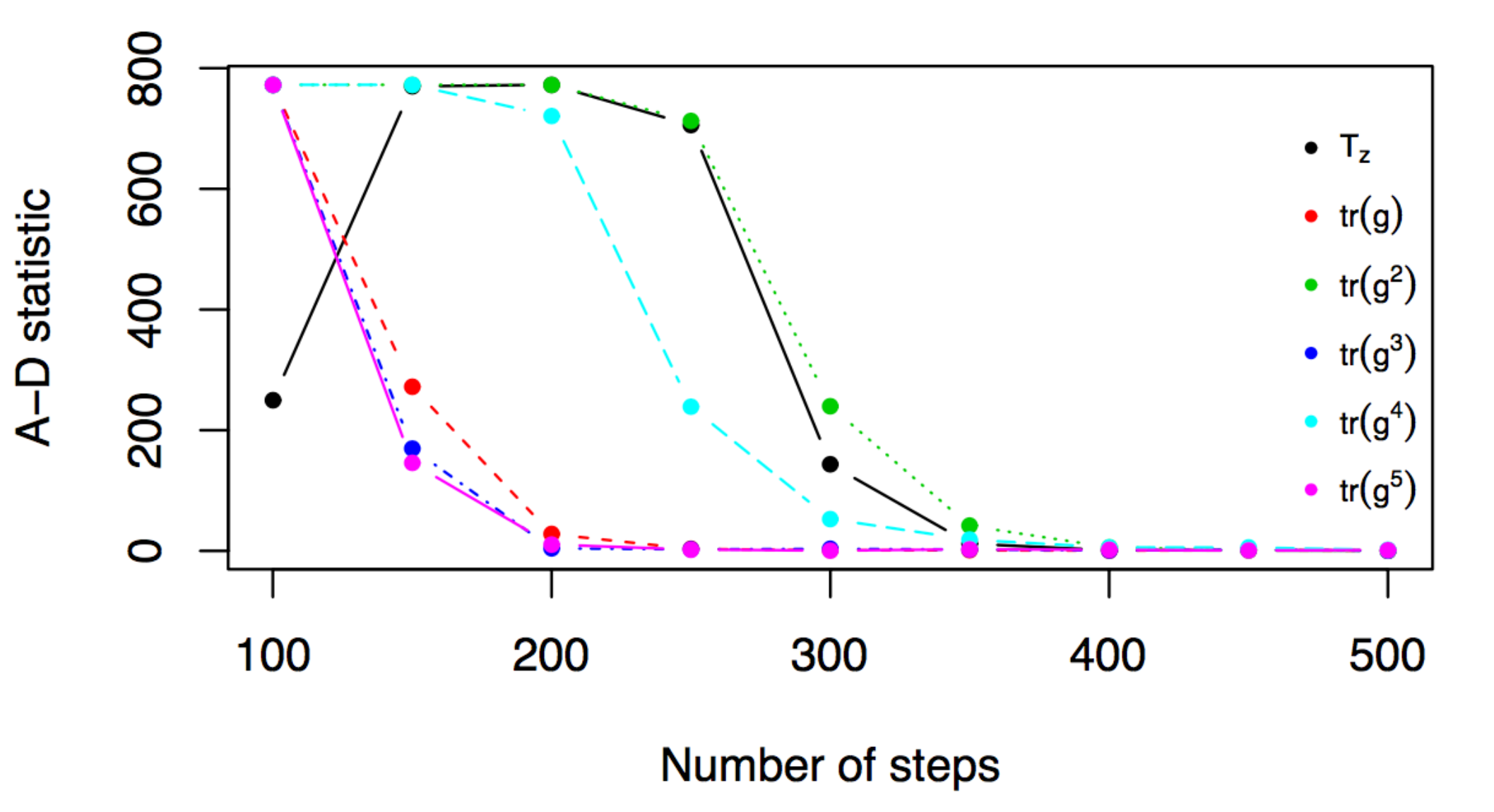}
 %   \caption*{{\scriptsize (a) Uniform sampling}}
  \end{minipage}
   \caption{Values of the Anderson-Darling statistic for comparison of the uniform sample and the Kac's walk for different tests.}
  \label{fig:ADValuesKacComp}
\end{figure} 

\end{remark}

The test based on $T_{1/2}^{(200)}$ applied to the new sampler, after a single iteration, provides no evidence for departure from the null; the Anderson-Darling p-value based on 1000 values of $T_{1/2}^{(200)}$ is equal to 0.35.

\subsubsection{Asymptotic distribution under the null hypothesis}
The representation (\ref{Def:EigenTest}) allows for derivation of the distribution under the null hypothesis of uniformity. This is the content of the following proposition.
\begin{prop}\label{NullAsymptoticDist}
Assume $g_1,\ldots, g_N$ are independent draws from the uniform distribution on $SO(2n+1)$, and $T_z^{(N)}$ is defined as in (\ref{Def:EigenTest}). Then, for any $z\in(0,1)$ one has
\begin{align*}
T_z^{(N)} \rightarrow T_z = \sum_{k=1}^\infty z^k \chi^2_{p(n,k)} \;\;\;(\text{as  }N\rightarrow \infty).
\end{align*}
On the right hand side, p(n,k) is the number of partitions of $k$ with at most $n$ parts and random variables $\chi^2_{p(n,k)}$ are chi-square with $p(n,k)$ degrees of freedom which are mutually independent.
\end{prop}
\begin{proof}
The CLT along with the orthogonality relations between the irreducible characters assert that $\sqrt{N}\widehat{\chi}_N(\lambda)$ converges to standard Gaussian variables for ${\lambda\neq0}$, and the limiting variables are mutually independent. Let $\mb{D} = \{\bm{x}= (x_\lambda)_{\lambda \neq 0}\mid \sum_{\lambda \neq 0} z^{|\lambda|} x_\lambda^2 < \infty \}$. $\mb{D}$ is a Hilbert space equipped with the inner-product
\begin{align*}
\langle \bm{x}, \bm{y} \rangle_{z} = \sum_{\lambda \neq 0} z^{|\lambda|} x_\lambda y_\lambda.
\end{align*}
Define the map $\phi : \mb{D} \mapsto \mb{R}$ as $\phi (\bm{x}) = \langle \bm{x},\bm{ x} \rangle_{z} =  \sum_{i \in \mc{I}} z^{|\lambda|}x_\lambda^2$; it is second order Hadamard directionally differentiable with derivative $\phi_{\bm{x}}^\prime (h) = 2\langle x,h\rangle_{z}$ and the second derivative $\phi^{\prime\prime}_{\bm{x}} (h) = 2 \phi (h)$. Let $\widehat{\chi}_N = (\widehat{\chi}_N(\lambda))_\lambda$. Then,
\begin{align*}
T_{z}^{(N)} = N \phi (\widehat{\chi}_N).
\end{align*}
\citet[Theorem 1.4.8]{van1996weak} asserts that weak convergence in a countable product space, to separable limiting variables, is determined by weak convergence of all finite-dimensional marginals. That means that $\sqrt{N} \widehat{\chi}_N$ converges weakly to a random element of $\mb{D}$, denoted by $Z = (Z_\lambda)_\lambda$ where $Z_\lambda$ are independent normal variables. Since, $\sqrt{N}(\widehat{\chi}_N-0) \rightarrow Z$, the second order Delta method \citep{romisch2005delta} yields
\begin{equation*}
N  \phi(\widehat{\chi}_N) \rightarrow \phi (Z) = \sum_{\lambda \neq 0} z^{|\lambda|} Z_\lambda^2.
\end{equation*}
 Collecting powers of $z$ proves the proposition.  
\end{proof}

\begin{remark}As a consequence of Proposition \ref{NullAsymptoticDist}, asymptotic expectation and variance of $T_z^{(N)}$ are given as 
\begin{align*}
 \mb{E}(T_z) = \sum_{k=1}^\infty z^k p(n,k)\quad \text{and} \quad
 \text{Var}(T_z) = 2\sum_{k=1}^\infty z^{2k} p(n,k).
\end{align*}
The right hand sides can be simplified using the identity \cite[Theorems 1.1 \& 1.4]{andrews1998theory}
\begin{align*}
\sum_{k=0}^\infty z^k p(n,k) = \prod_{i=1}^n \frac{1}{1-z^i}.
\end{align*}
We get
\begin{align*}
\mb{E}(T_z) = \prod_{i=1}^n \frac{1}{1-z^i}-1\quad \text{and} \quad
 \text{Var}(T_z) = 2\left(\prod_{i=1}^n \frac{1}{1-z^{2i}} -1\right).
\end{align*}
\end{remark}
Table \ref{table:AsymptoticMeanEigenTest} shows the asymptotic mean and variance for $51 \times 51$ rotations, i.e.\ $n=25$, and different values of $z$. Finite sample expectation and variance under Setup \ref{setup} and $z=0.5$ are $2.48$ and $1.26$ respectively. Empirical quantiles are given in Table \ref{table:EmpiricalQuantilesEigenTest}.
\begin{table}[h!]
\caption{Mean and variance of $T_z$ for $n = 25$.}
\label{table:AsymptoticMeanEigenTest}
\centering
%\resizebox{\columnwidth}{!}{%
\begin{tabular}{lcccc}
\hline
$z=$ & \multicolumn{1}{c}{\begin{tabular}{@{}c@{}}0.5\end{tabular}} & \multicolumn{1}{c}{\begin{tabular}{@{}c@{}}0.8\end{tabular}} & \multicolumn{1}{c}{\begin{tabular}{@{}c@{}}0.9\end{tabular}}&\multicolumn{1}{c}{\begin{tabular}{@{}c@{}}0.99\end{tabular}}\\
\hline
 mean & 2.46 & 291.45 &402914.7& $2.844628 \times 10^{25}$ \\
 variance & 0.9047073 & 18.86372 & 870.2173 & $8.097291\times 10^{18}$ \\
 \hline
\end{tabular}
%}
\end{table}
\begin{table}[h!]
\caption{$P\left(T_z^{(N)} \le w_p\right) = p $ based on 1000 Monte Carlo runs with $N=200, n= 25, z= 0.5$.}
\label{table:EmpiricalQuantilesEigenTest}
\centering
%\resizebox{0.8\columnwidth}{!}{%
\begin{tabular}{lccccccccc}
\hline
$p$ & \multicolumn{1}{c}{\begin{tabular}{@{}c@{}}0.01\end{tabular}} &\multicolumn{1}{c}{\begin{tabular}{@{}c@{}}0.05\end{tabular}} & \multicolumn{1}{c}{\begin{tabular}{@{}c@{}}0.10\end{tabular}} & \multicolumn{1}{c}{\begin{tabular}{@{}c@{}}0.25\end{tabular}}&\multicolumn{1}{c}{\begin{tabular}{@{}c@{}}0.50\end{tabular}}&\multicolumn{1}{c}{\begin{tabular}{@{}c@{}}0.75\end{tabular}}&\multicolumn{1}{c}{\begin{tabular}{@{}c@{}}0.90\end{tabular}}&\multicolumn{1}{c}{\begin{tabular}{@{}c@{}}0.95\end{tabular}}&\multicolumn{1}{c}{\begin{tabular}{@{}c@{}}0.99\end{tabular}} \\
\hline
 $w_p$ & 0.93& 1.16 & 1.30 &1.66& 2.20& 2.94 & 3.94&4.65 & 7.38\\
 \hline
\end{tabular}
%}
\end{table}

\subsubsection{Distribution under fixed alternative hypotheses}
The alternative distribution is given below.
\begin{prop}Let $F$ be a probability measure on $[0,\pi]^n$ which is different from $f_{1,\frac{3}{2},\frac{1}{2}}$. Let $\theta^{(1)},\ldots,\theta^{(N)}$ be independent draws from $F$. Then, $T_z^{(N)}$ is asymptotically normal. In fact,
\begin{align*}
\sqrt{N}(T_z^{(N)}/N - \mu) \rightarrow \mc{N}(0,\sigma^2) \quad \text{ as } N\rightarrow \infty,
\end{align*}
with $\mu = \int r^2(\theta) F(d\theta) - 1$ and $\sigma^2 = 4\left[\int\left(\int r(\theta)g(\theta,\phi)f_{1,\frac{3}{2},\frac{1}{2}}(d\theta)\right)^2F(d\phi)-\mu^2 \right] $, where $g$ and $r$ are defined as
$r(\theta)=\int g(\theta,\phi)F(d\phi)$ and
\begin{align*}
g(\theta,\phi) = \frac{(1-\sqrt{z})^n \det\left[ \frac{(1+\sqrt{z})^2+2\sqrt{z}(\cos\theta_k+\cos\phi_l)}
{(1+z)^2 - 4\sqrt{z}(1+z)\cos \theta_k \cos \phi_l+2z (\cos 2\theta_k +\cos 2\phi_l)}
\right]_{k,l}}{(16z)^{\binom{n}{2}/2}\prod_{i<j} \left(\cos\theta_i-\cos\theta_j\right) \prod_{i<j} 
\left(\cos\phi_i-\cos\phi_j\right)}.
\end{align*}
\end{prop}
\begin{proof} Proof follows from Proposition (4.6) of \cite{gine1975invariant} and is given in detail in section 4 of the supplementary material \cite{sepehri2017supplement}.
\end{proof}
\begin{remark}
A direct consequence is that $T_z^{(N)}$ is consistent against all fixed alternatives; not only the limiting distribution differs, so does the scaling. In particular, $\mb{P}_\nu \left(T_z^{(N)} > c_{z,1-\alpha}\right) \rightarrow 1$ as $N$ tends to infinity, for all alternatives $\nu$.
\end{remark}

\section{Beyond the Eigenvalues}\label{Sec:TestBeyondEigenValues}
Although the test based on $T_z^{(N)}$ proved successful in different examples, it failed to reject the null hypothesis against the alternative given by the new sampler of \citet{jones2011randomized}, even after only one step of the sampler. To overcome this deficiency, it is needed, and natural, to resort to the properties beyond the eigenvalues. This section presents a test for the full Haar measure on $SO(2n+1)$ based on the machinery of section \ref{Sec:SpectralTestsGeneralSpaces}.

Let $G = SO(2n+1)$. Given data $g_1, \ldots, g_N \in G$ independently drawn from a measure $\nu$ on $G$, consider testing the null hypothesis $H_0: \nu = \mu$, where $\mu$ is the uniform (Haar) measure. To construct a spectral test, use the orthonormal basis for $\mc{L}^2(G)$ given by the matrix coordinates of the irreducible representations; see section 1 in \cite{sepehri2017supplement}. Define the Fourier component corresponding to $\lambda$ as
\begin{align*}
\widehat{\pi}_N(\lambda) = \frac{1}{N} \sum_{i=1}^N \pi_\lambda (g_i).
\end{align*}
Note that $\widehat{\pi}_N(\lambda)$ is a $d_\lambda\times d_\lambda$ matrix. A test that rejects for large values of 
\begin{align}\label{Def:GeneralFullTest}
U_{\bm{c}}^{(N)} = N \sum_{\lambda\neq 0} c_\lambda \| \widehat{\pi}_N(\lambda)\|_F^2
\end{align}
can be used, for an arbitrary sequence of positive weights $\bm{c} = \{c_\lambda\}$. To find a closed form for $U_{\bm{c}}$, note that
\begin{align*}
 \|\widehat{\pi}_N(\lambda)\|_{F}^2 &= \Tr(\widehat{\pi}_N(\lambda)\widehat{\pi}_N(\lambda)^\ast)
 												= \Tr(\frac{1}{N^2} \sum_{i,j} \pi_\lambda(g_i)\pi_\lambda(g_j)^\ast)\\
 												&= \Tr(\frac{1}{N^2} \sum_{i,j} \pi_\lambda(g_i)\pi_\lambda(g_j^{-1}))
 												= \Tr(\frac{1}{N^2} \sum_{i,j} \pi_\lambda(g_i g_j^{-1}))\\
 												&= \frac{1}{N^2} \sum_{i,j} \Tr( \pi_\lambda(g_i g_j^{-1}))
 												= \frac{1}{N^2} \sum_{i,j} \chi_\lambda(g_i g_j^{-1}),
 \end{align*}
where the third equality follows from the fact that $\pi_\lambda$ is a unitary matrix. Therefore, one can write
\begin{align*}
U_{\bm{c}}^{(N)} = \frac{1}{N} \sum_{i,j} \sum_{\lambda\neq 0} c_\lambda \chi_\lambda(g_i g_j^{-1}).
\end{align*}
Using tools from representation theory of $SO(2n+1)$, a closed from expression for $\sum_\lambda c_\lambda \chi_\lambda(h)$ can be found for $\bm{c}$ of a particular form as follows; see section 3 in the supplementary material \cite{sepehri2017supplement}. For arbitrary parameters $z,q\in (0,1)$, there exists a set of positive weights $c_\lambda(z,q) $, given explicitly by equation (24) in the supplementary material \cite{sepehri2017supplement}, such that
\begin{align}\label{Eq:ClosedFormFullTest}
\sum_{\lambda \neq 0} c_\lambda(z,q) \chi_\lambda(g) = \frac{\prod_{k<j}(1-z^2 q^{k+j-2}) \prod_k (1+z q^{k-1})}{\prod_{k,j} (1-z q^{k-1}e^{i \theta_j} )(1-z q^{k-1} e^{-i \theta_j})} - 1,
\end{align}
where $1, e^{\pm i \theta_1},\ldots,e^{\pm i \theta_n}$ are eigenvalues of $g$. This motivates the following definition.
\begin{defin}\label{Defin:FullTest} Let $z,q \in (0,1)$ be arbitrary parameters. Define the test statistic as
\begin{align}
\begin{split}
U_{z,q}^{(N)} 
&= \frac{1}{N}\sum_{k,l=1}^N  \left(\frac{\prod_{i<j}(1-z^2 q^{i+j-2}) \prod_i (1+z q^{i-1})}{\prod_{i,j} (1-z q^{i-1} e^{i \theta^{k,l}_j})(1-z q^{i-1} e^{-i \theta^{k,l}_j})} -1\right),
\end{split}
\end{align}
where $1,e^{\pm i \theta^{k,l}_1},\ldots, e^{ \pm i \theta^{k,l}_n}$ are the eigenvalues of $g_k g_l^T$.
\end{defin}
\begin{remark} The test based on $U_{z,q}^{(N)}$ fits into the framework proposed by \citet{gine1975invariant}. In fact, the eigenfunction of the Laplace-Beltrami operator on $G$ are exactly the matrix coordinates of the irreducible representations of $G$, for all compact classical groups.
\end{remark}

The test based on $U_{z,q}^{(N)}$ was applied to the new sampler of \citet{jones2011randomized} and exhibited non-trivial power against the alternative distribution generated by a single iteration of the sampler. For $z=0.2, q= 0.4, n = 25$, and $N=200$ the histogram of 1000 values of $U_{z,q}^{(N)}$ is illustrated in Figure \ref{fig:HisFullTestRokh}. 
The power of the $5\%$-level test is 0.17 which is non-trivial although not particularly high. However, this might well be a result of the small sample size. Note that the testing problem is a non-parametric test of goodness-of-fit in $\binom{51}{2} = 1275$ dimensions with only $200$ observations. Note that recent work of \citet{arias2016remember} suggest that in the non-parametric goodness-of-testing problem in $d$ dimensions, the sample size $N$ has to be exponential in $d$ in order to have a test that discriminates against the alternatives of fixed distance from the null; that is, $\log (N) / d$ should be bounded below. In the case of $SO(51)$, $d=1275$; even $N = 200000$ results in $\log N / d \approx 0.0095$. For a brief description of their results see section 8 in the supplementary material \cite{sepehri2017supplement}.

\begin{figure}[h!]
  \centering
  \begin{minipage}[b]{0.6\textwidth}
    \includegraphics[width=\textwidth]{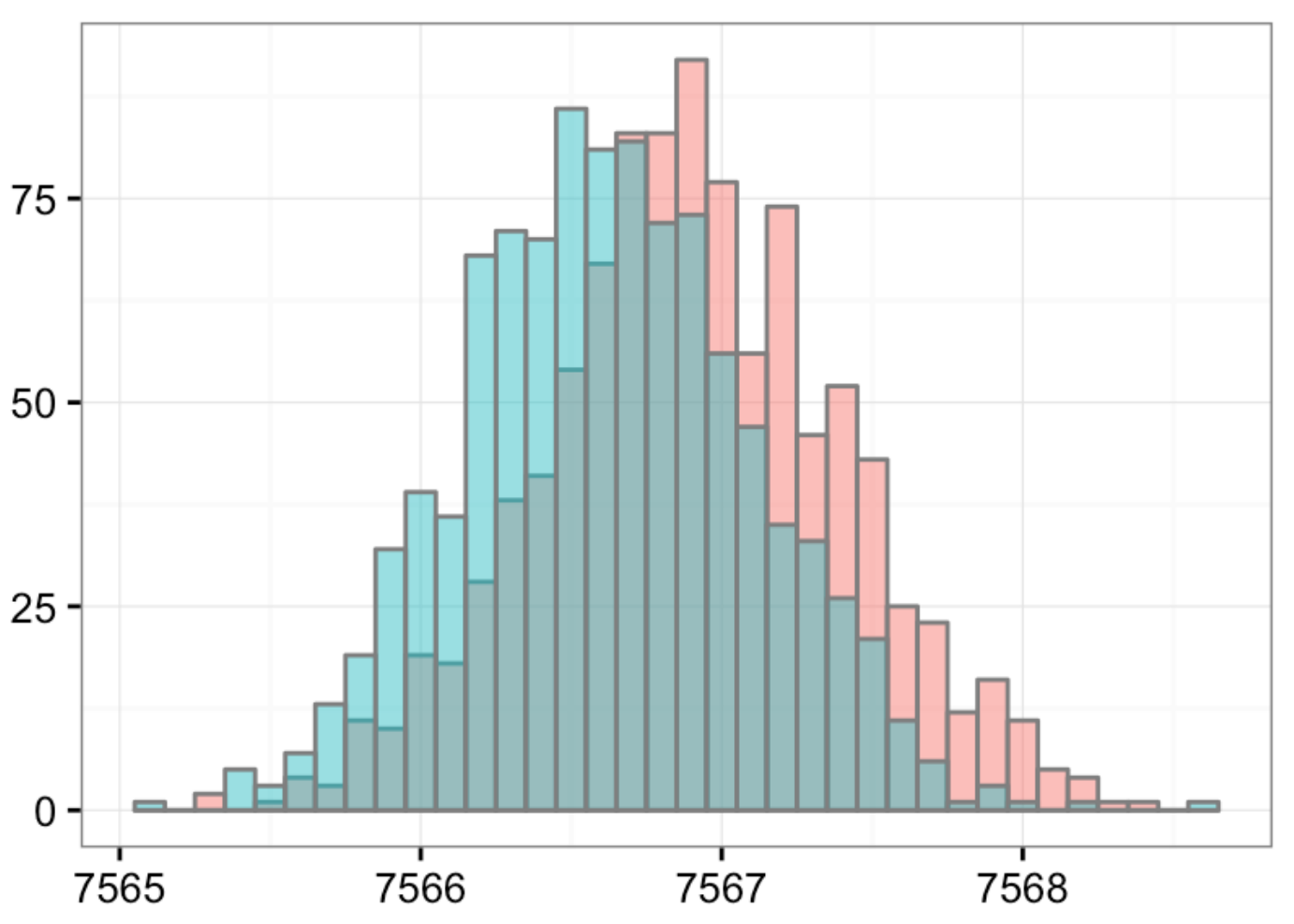}
 %   \caption*{{\scriptsize (a) Uniform sampling}}
  \end{minipage}
   \caption{Histogram of $U_{z,q}^{(N)}$ under Setup \ref{setup}, $ z= 0.2$, and $ q =0.4$. The alternative is the distribution of the orthogonal matrices generated by a single iteration of the new sampler.}
  \label{fig:HisFullTestRokh}
\end{figure}

Under the same setup as above the 1000 values under the null and alternative were compared using Anderson-Darling test, where alternative was taken to be the output of the new sampler after different number of iterations. The result is shown in Figure \ref{fig:AdKsValues} and Table \ref{table:AdKsPvalRokh}. It appears that the distribution of the output is not close to uniform after only one iteration. The null hypothesis is rejected at $5\%$ level, suggesting some departure from uniformity, for the distribution of the output after two iterations, but not beyond two steps. This is in agreement with the prescribed number of steps given in \citet{jones2011randomized} for the chain to mix.
\begin{figure}[h!]
  \centering
  \begin{minipage}[b]{0.45\textwidth}
    \includegraphics[width=\textwidth]{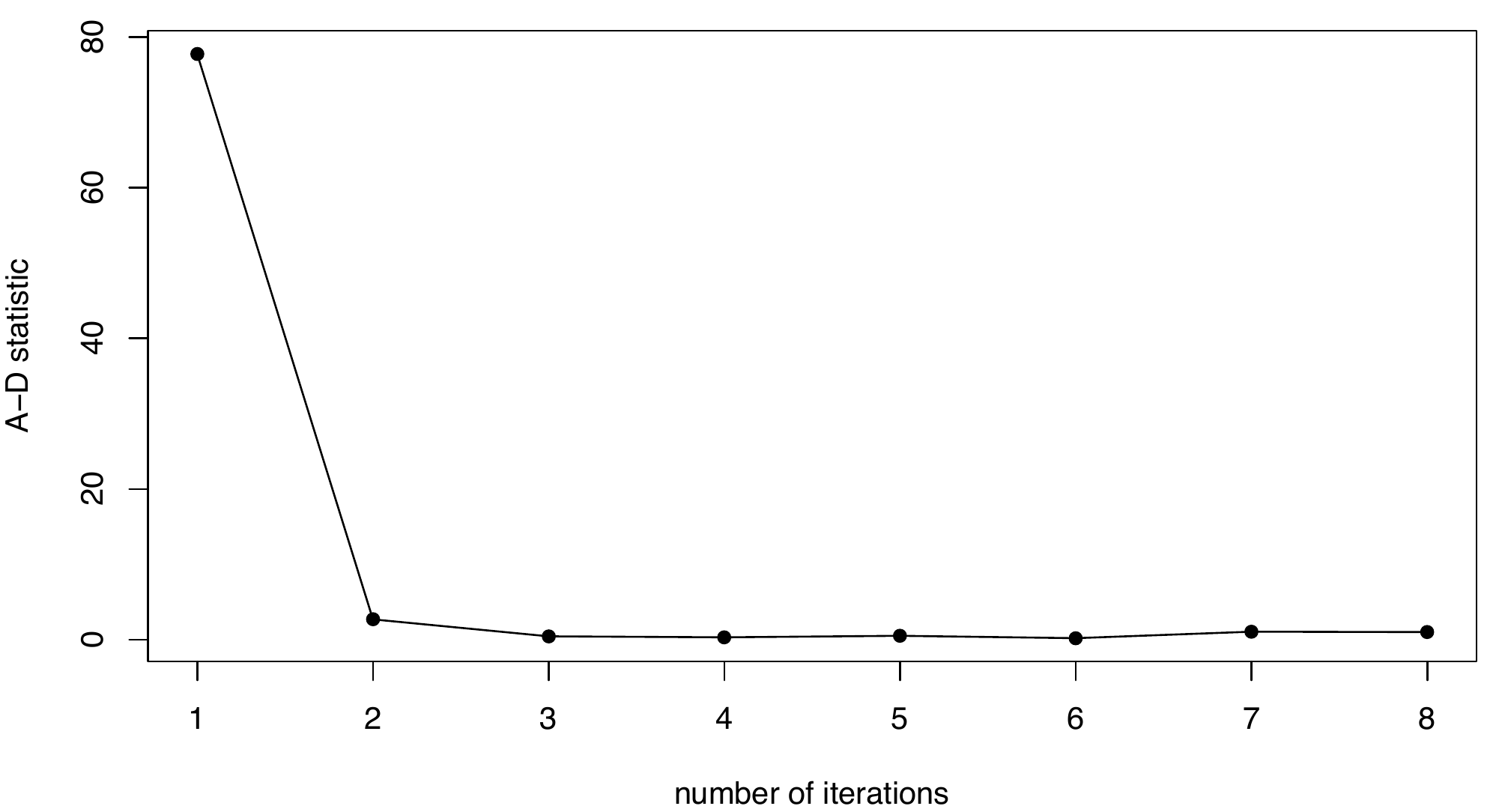}
 %   \caption*{{\scriptsize (a) Uniform sampling}}
  \end{minipage}
   \caption{Values of the Anderson-Darling statistic versus the number of iterations of the new sampler.}
  \label{fig:AdKsValues}
\end{figure} 

\begin{table}[h!]
\caption{$p$-values corresponding to different iterations of the new sampler}
\label{table:AdKsPvalRokh}
\centering
\resizebox{\columnwidth}{!}{%
\begin{tabular}{lcccccccc}
\hline
$\#$ of iterations & \multicolumn{1}{c}{\begin{tabular}{@{}c@{}}1\end{tabular}}& \multicolumn{1}{c}{\begin{tabular}{@{}c@{}}2\end{tabular}} & \multicolumn{1}{c}{\begin{tabular}{@{}c@{}}3\end{tabular}} & \multicolumn{1}{c}{\begin{tabular}{@{}c@{}}4\end{tabular}}&\multicolumn{1}{c}{\begin{tabular}{@{}c@{}}5\end{tabular}} &\multicolumn{1}{c}{\begin{tabular}{@{}c@{}}6  \end{tabular}} &\multicolumn{1}{c}{\begin{tabular}{@{}c@{}}7\end{tabular}}& \multicolumn{1}{c}{\begin{tabular}{@{}c@{}}8\end{tabular}}\\
\hline
  A-D test & 5.94e-43 & 3.92e-02 & 8.02e-01 & 9.25e-01 & 7.29e-01 & 9.90e-01 & 3.26e-01 & 3.52e-01\\
 \hline
\end{tabular}
}
\end{table}

\begin{remark} Power of the test based on $U^{(N)}_{z,q}$ indeed depends on the choice of $(z,q)$. Given the form of $c_{\lambda}(z,q)$, the test with a larger value of $z$ has more power against the alternatives which deviate from the uniform distributions in high frequencies. On the other hand. when the alternative differs from the uniform distribution on lower frequencies the test with smaller $z$ is more powerful. Our preliminary numerical investigations suggest that the alternatives considered in the present paper fall into latter category.

Dependence on $q$ is less clear, because the coefficient $c_\lambda (z,q)$'s relation to $q$ is more complicated. Figure \ref{fig:Weightsq} illustrates $\log(c_{\lambda}(1,q))$ for different values of $q$ and several partitions $\lambda$. $c_\lambda$ is increasing in $q$ but order of the coefficients for a fixed $q$ depends on $q$ which makes it hard to draw general conclusions.
\begin{figure}[h!]
  \centering
  \begin{minipage}[b]{0.435\textwidth}
    \includegraphics[width=\textwidth]{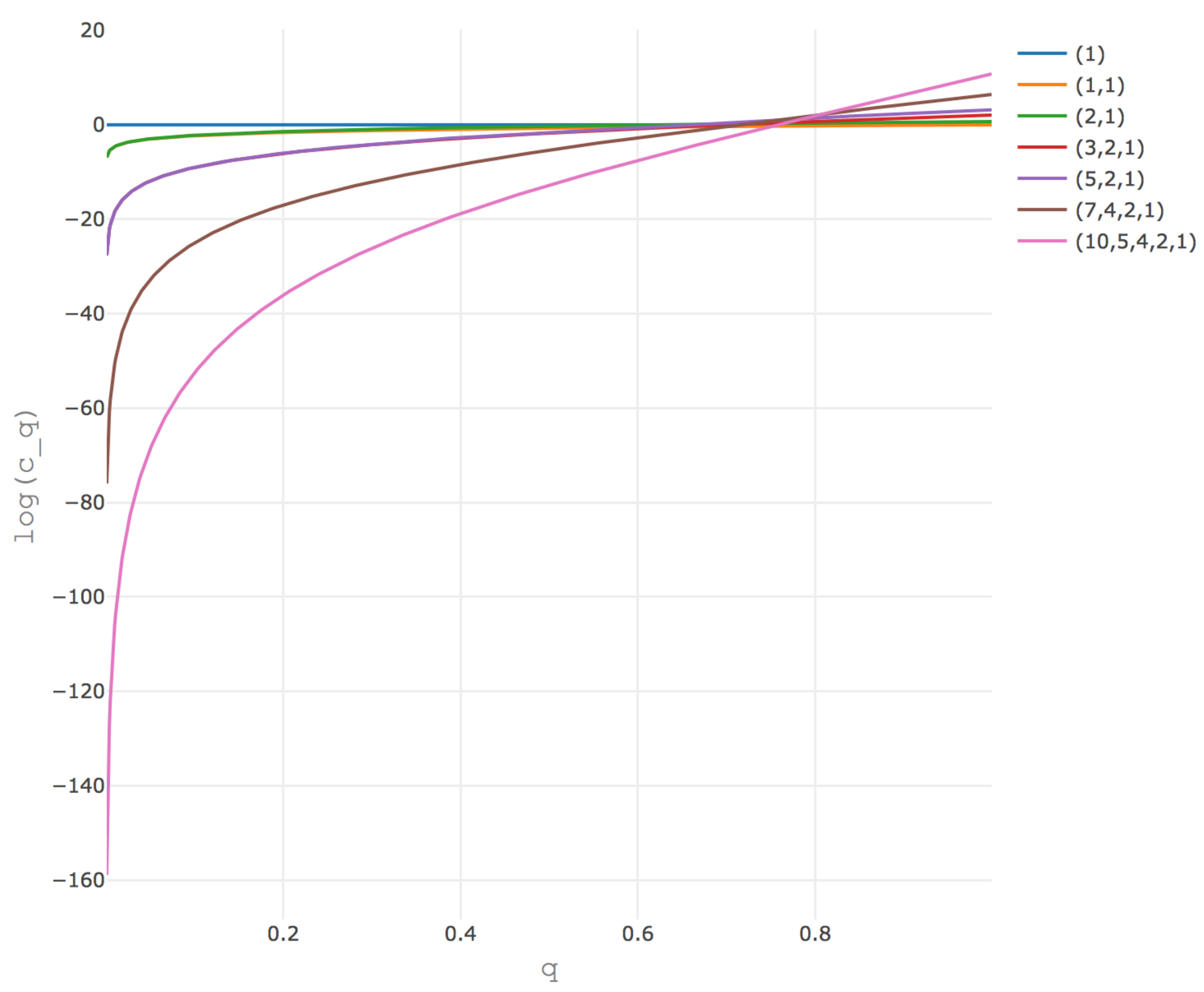}
 %   \caption*{{\scriptsize (a) Uniform sampling}}
  \end{minipage}\hfill
  \centering
  \begin{minipage}[b]{0.45\textwidth}
    \includegraphics[width=\textwidth]{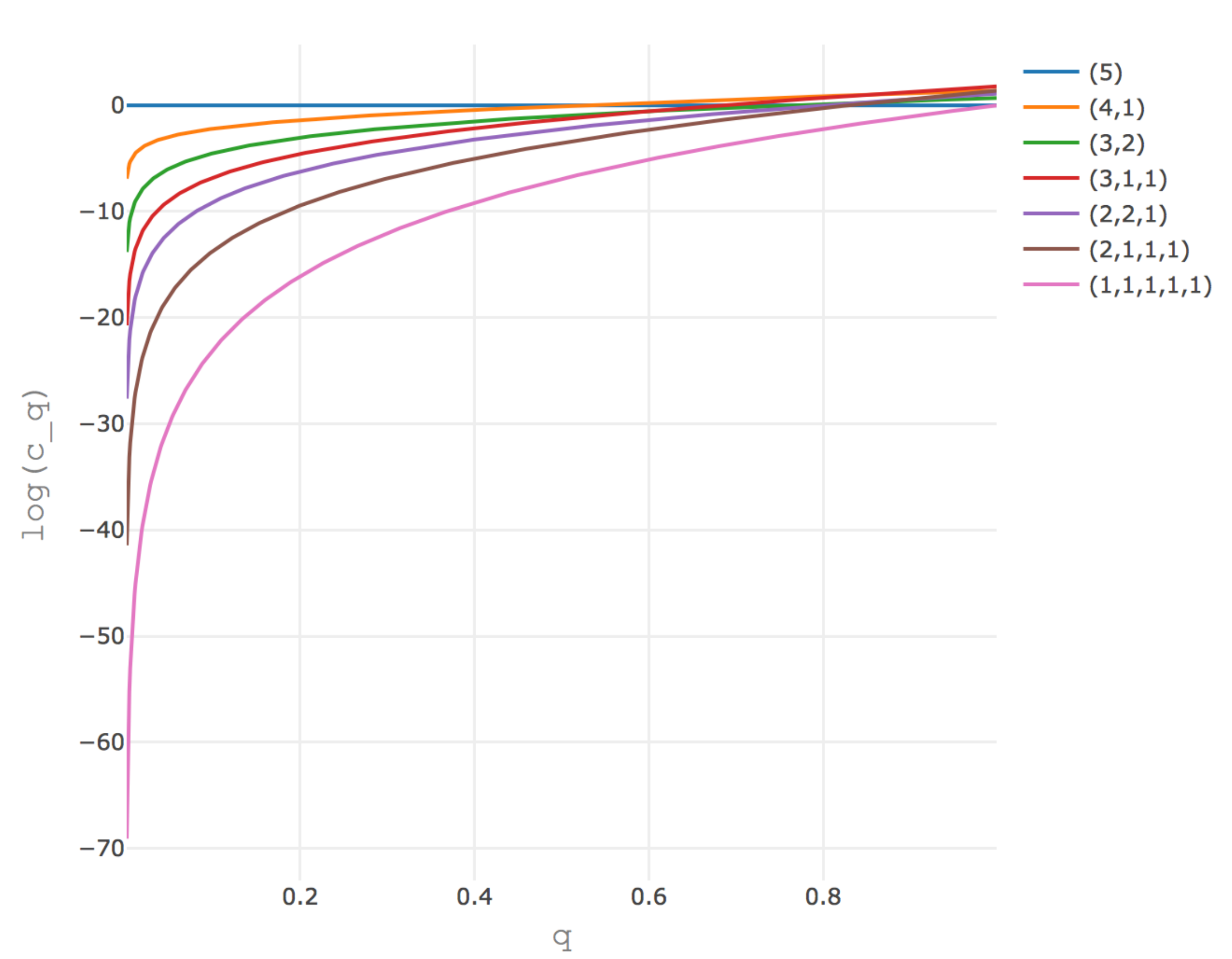}
 %   \caption*{{\scriptsize (a) Uniform sampling}}
  \end{minipage}
   \caption{Logarithmic graph of $c_\lambda(1,q)$ against $q$ for several $\lambda$. The left plot shows the graphs for various frequencies. The plot on the right shows the graphs for all $|\lambda| = 5$.}
  \label{fig:Weightsq}
\end{figure} 
\end{remark}

\subsection{Distribution under the null and alternative}\label{Sec:AsymNullAltFullTest}
The asymptotic distribution of $U_{z,q}^{(N)}$ under the null and fixed alternatives can be derived in a similar fashion to those of $T_z^{(N)}$.
\begin{prop}[Asymptotic null distribution]  Assume $g_1,\ldots, g_N$ are independent draws from the uniform distribution on $SO(2n+1)$,and $z, q\in(0,1)$. Then,
\begin{align*}
U_{z,q}^{(N)} \rightarrow U_{z,q} = \sum_{\lambda\neq 0} \frac{c_\lambda (z,q)}{d_\lambda} \chi^2_{d_\lambda^2},
\end{align*}
where $d_\lambda = \chi_\lambda (I)$ is the dimension of the irreducible representation corresponding to $\lambda$, $c_\lambda (z,q)$ is as in (\ref{Eq:ClosedFormFullTest}), and the chi-square variables are mutually independent.
\end{prop}
\begin{proof}
The statement follows from the orthogonality relations between the matrix-coordinates of the irreducible representations, the central limit theorem, and the fact that
\begin{align*}
\mb{E}[(\pi^\lambda_{ij})^2] = 1/d_\lambda.
\end{align*}
\end{proof}
The asymptotic distribution under the alternative is given as follows.
\begin{prop}[Asymptotic alternative distribution] Let $F$ be a distribution on $SO(2n+1)$ different from the uniform measure. Given data $g_1,\ldots,g_N$ independently drawn from $F$, $U_{z,q}^{(N)}$ is asymptotically normally distributed. In fact,
\begin{align*}
\sqrt{N}(U_{z,q}^{(N)}/N - \upsilon) \rightarrow \mc{N}(0,\sigma^2) \quad \text{ as } N\rightarrow \infty,
\end{align*}
with $\upsilon = \int r^2(g) F(dg)-1$ and $\sigma^2 = 4\left[\int\left(\int r(g)u(g,h)\mu(dg)\right)^2F(dh)-\upsilon^2 \right] $, where $u$ is defined below in (\ref{UNALTproof}) and $r$ is defined as
$r(g)=\int u(g,h)F(dh)$
\end{prop}
\begin{proof}
Proof follows from Proposition (4.6) of \cite{gine1975invariant} and the following lemma. The lemma is proved in section 4 of the supplementary material \cite{sepehri2017supplement}
\begin{lemma} For $g_1,\ldots, g_N \in SO(2n+1)$ one has 
\begin{align}\label{UNexp}
U_{z,q}^{(N)} = \frac{1}{N} \int \left| \sum_{i=1}^N u(g_i,g) \right|^2 \mu(dg),
\end{align}
where $u$ is defined through
\begin{align}\label{UNALTproof}
u(g,h) = \sum_{\lambda\neq 0} \sqrt{d_\lambda c_\lambda (z,q)}\chi_\lambda (g^T h).
\end{align}
\end{lemma}
\end{proof}
\begin{remark}
A direct consequence is that $U_{z,q}^{(N)}$ is consistent against all fixed alternatives; not only the limiting distribution differs, so does the scaling.
\end{remark}

\begin{remark} The associate editor brought to our attention a recent work of \citep{kerkyacharian2012concentration} that provides concentration inequalities and confidence band for a class of needlet density estimators on compact homogeneous manifolds, in particular on compact classical groups. Without getting into details, the following is a high level description of their approach. They introduce a needlet projection kernel $A_j(x,y)$ of order $j$ for any non-negative integer $j$. Then, given observations $x_1, \ldots , x_n$ from a density $f$, they define a needlet density estimator through
\begin{align*}
f_n (j,y) = \frac{1}{n} \sum_{i=1}^n A_j(x_i , y).
\end{align*}
Provided that $f$ is bounded they prove the following concentration inequality.
\begin{prop}[Proposition 4 in \citep{kerkyacharian2012concentration}] Let $\mathbf{M}$ be a compact homogeneous manifold and suppose $f : \mathbf{M} \rightarrow [ 0 , \infty)$ is bounded. We have for every $n \in \mb{N}$, every $j \in \mb{N}$, every $\Omega \subset \mathbf{M}$, and every $x \in \mathbf{M}$ that 
\begin{align*}
\mb{P} \left( \sup_{y \in \Omega} | f_n(j ,y ) - \mb{E} f_n(j ,y )  |  \ge \sigma^R (\Omega , n , j , x)\right) \le e^{-x},
\end{align*}
where $ \sigma^R (\Omega , n , j , x) $ depends on $\mathbf{M}, n , j , \Omega \text{and} x$, and is defined explicitly in their paper. 
\end{prop}
In the case of $f$ being the uniform distribution, $\mb{E} f_n(j ,y ) = f = \text{constant}$. Therefore, the concentration inequality above gives a confidence band for the density estimator $f_n(j,y)$. In particular, this confidence band can be used to construct \textbf{non-asymptotic} tests of goodness-of-fit for the uniform distribution on compact classical groups. The constants in the definition of $\sigma^R (\Omega , n , j , x)  $ are computed explicitly for $SO(3)$ in the Supplementary material.
\end{remark}

\section{Numerical Comparison of Different Tests}\label{Sec:Numeric}
This section compares different tests discussed in this paper on the benchmark examples, with a particular focus on detection of the cutoff, as well as the new sampler. Setup \ref{setup} is considered ($n=51$). Focus on the following four test: Rayleigh's test, Gine's test, $T_{z}^{(N)}$, and $U_{z,q}^{(N)}$. The numerical observations are summarized below.
\subsubsection*{The benchmark examples}
Each test was computed 1000 times on the samples generated by the benchmark examples for different number of steps. Each of the 1000 simulations were based on $N=200$ observations; $T_z^{(N)}$ was computed with $z=0.5$ and $U_{z,q}^{(N)}$ with $z=0.2$ and $q=0.4$. The samples were generated using $k$ steps for
\begin{align*}
k \in \{100,150,200,250,300,350,400,450,500\}
\end{align*}
for the Kac's walk and 
\begin{align*}
k \in \{50,75,90,100,110,125,140,150,175,200\}.
\end{align*}
for product of random reflections.
For each fixed number of steps the 1000 values were compared to those corresponding to the uniform distribution using the Anderson-Darling test. Figure \ref{fig:ADValuesKacALLtests} illustrates the values are plotted against the number of steps of the chain.
\begin{figure}[h!]
  \centering
  \begin{minipage}[b]{0.48\textwidth}
    \includegraphics[width=\textwidth]{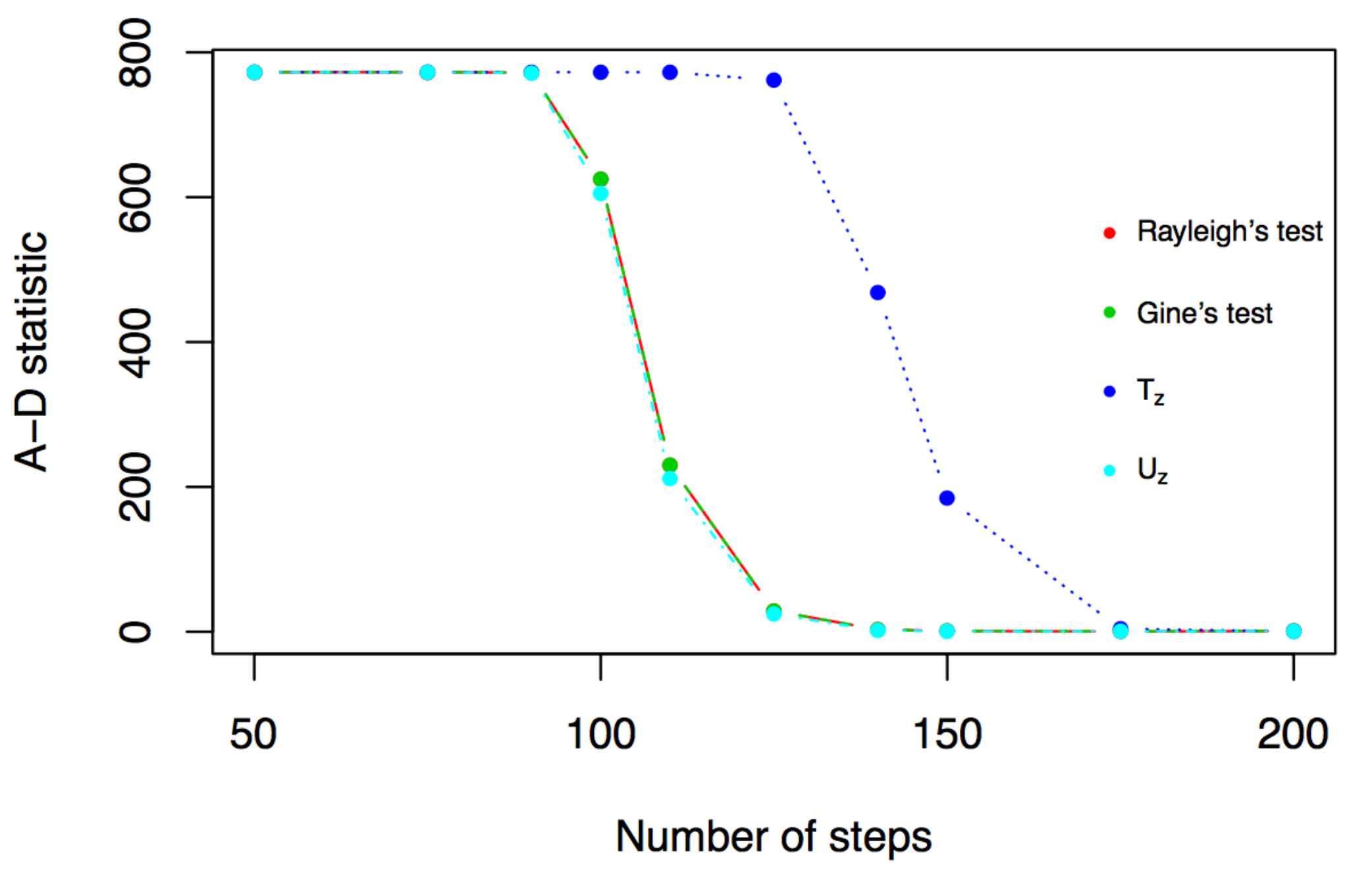}
 %   \caption*{{\scriptsize (a) Uniform sampling}}
  \end{minipage}
  \hfill
  \begin{minipage}[b]{0.48\textwidth}
    \includegraphics[width=\textwidth]{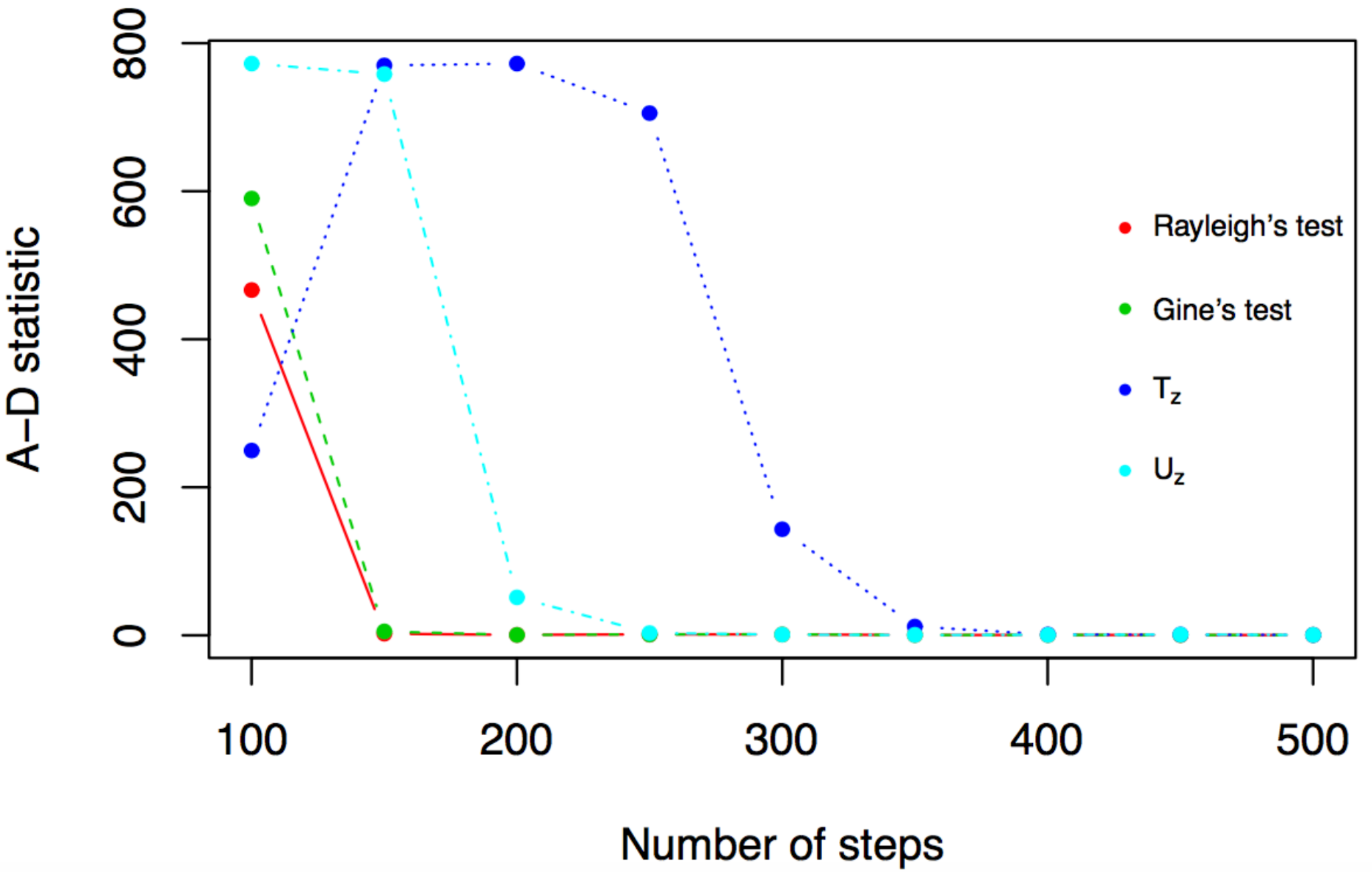}
%    \caption*{{\scriptsize (b) Naive sampling}}
  \end{minipage}
   \caption{Values of the Anderson-Darling statistic for comparison of the uniform sample and the product of random reflections (left) and the Kac's walk (right).}
  \label{fig:ADValuesKacALLtests}
\end{figure} 

Figure \ref{fig:ADValuesKacALLtests} suggests that the Gine's test has the least power against the alternative generated be the Kac's walk among the four tests considered here. The Rayleigh's test and $U_{z,q}^{(N)}$ seem to perform similarly, indicating some evidence for a cutoff but, perhaps, earlier than it should possibly occur. The test based on $T_z^{(N)}$ outperforms the other three tests and provides evidence that a cutoff does not occur with less than 350 step, if it occurs at all.

A similar but slightly different result holds for product of random reflection; the Rayleigh's test has the least power, the Gine's test and $U_{z,q}^{(N)}$ are qualitatively identical. Again, $T_z^{(N)}$ is superior to the other three tests; it picks up the occurrence of the cutoff and suggests that it might happen in around 175 steps.
\subsubsection*{The new sampler}
The same procedure was repeated for $1\le k \le 8$ iterations of the new sampler of \citet{jones2011randomized}.
The result is shown in Figure \ref{fig:ADValuesRokhALLtests}.
\begin{figure}[h!]
  \centering
  \begin{minipage}[b]{0.7\textwidth}
    \includegraphics[width=\textwidth]{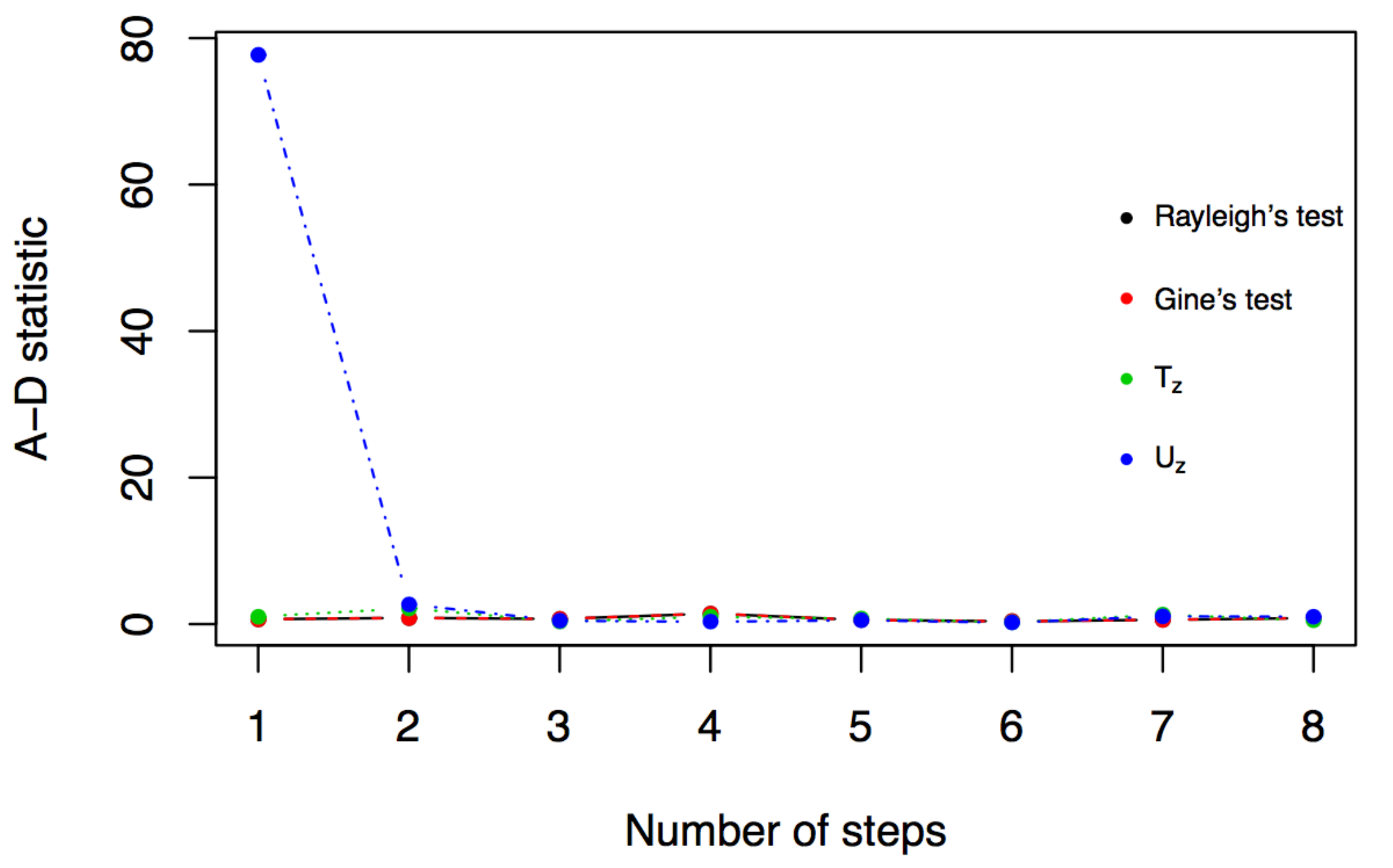}
 %   \caption*{{\scriptsize (a) Uniform sampling}}
  \end{minipage}
   \caption{Values of the Anderson-Darling statistic for comparison of the uniform sample and the new sampler.}
  \label{fig:ADValuesRokhALLtests}
\end{figure} 

The Rayleigh's test, the Gine's test, and the test based on $T_z^{(N)}$ exhibit no power against this alternative. Only $U_{z,q}^{(N)}$ has power against the alternative generated by a single iteration of the sampler. It also rejects the null hypothesis at $5\%$-level for the sample generated by two iterations of the sampler. It, too, has no power beyond two iterations. This is in agreement with the recommendation of \citet{jones2011randomized} on the number of iterations needed for mixing of the sampler.

\section{Asymptotic Properties Under Local Alternatives}\label{Sec:LocalAsymp}
This section studies properties of $T_z^{(N)}$ against local alternatives, i.e.\ alternatives that approach the null as the sample size grows. The case of $U_{z,q}^{(N)}$ is similar. Section 7 in the supplementary material \cite{sepehri2017supplement} presents a brief introduction to Le Cam's theory of \emph{asymptotically normal experiments}, which is the framework used here to study the local properties, as well as a detailed analysis of the local properties of the tests of this paper.

\subsection{Power calculations under local alternatives}\label{Sec:LocalPowerEigTEstFINITEDim}
The following standard local hypothesis testing setting is considered here. Let $\{f(\cdot\mid \theta) \mid \theta \in \Omega\}$ be a Q.M.D. family of density functions with respect to the eigenvalue distribution induced by the Haar measure, where $\Omega \subset \mb{R}^k$ for a fixed $k$. Assume that $f(\cdot\mid \theta_0) = 1$, that is, $\theta_0$ corresponds to $f_{1,\frac{3}{2},\frac{1}{2}}$. Given data $x_1,\ldots, x_N \sim f(\cdot\mid \theta)$, consider testing $H_0: \theta = \theta_0$ against $H_1: \theta = \theta_0+ h/\sqrt{N}$ for a fixed $h \in \Omega$. Let $\ell(x\mid \theta) = \log f(x\mid \theta)$ be the \textit{log-likelihood function}, $\eta (x\mid \theta) = \nabla_\theta \ell (x\mid \theta) $ the \textit{score function}, and $\mc{I}(\theta) = - \mb{E}_\theta \nabla_\theta^2 \ell (x\mid \theta)$ the \textit{Fisher information matrix} at $\theta$. Let $L_N$ denote the logarithm of the likelihood ratio of the data; Le Cam's first lemma \cite[Theorem 12.2.3]{lehmann2006testing} asserts
\begin{align}\label{Eq:LogLikelihoodBasicEig}
\begin{split}
L_N &= \sum_{i=1}^N \ell(x_i\mid \theta_0 + h/\sqrt{N}) - \ell(x_i\mid \theta_0)\\
		& = \left(\frac{ \sum_{i=1}^N \eta (x_i\mid \theta_0)}{\sqrt{N}}\right)^T h + \frac{1}{2} h^T \left( \frac{\sum_{i=1}^N \nabla_\theta^2 \ell (x_i\mid \theta_0)}{N} \right)h + o_p(1).
\end{split}
\end{align}
Since the score function of QMD families is square-integrable, one has
\begin{align}\label{Eq:L2ExpansionScoreEig}
\eta (\cdot \mid \theta_0) = \sum_\lambda \widehat{\eta}(\lambda) \chi_\lambda(\cdot),
\end{align}
where $\widehat{\eta}(\lambda)$ is the Fourier coefficient $\int \eta(x\mid \theta_0) \chi_\lambda (x) dx$ and the equality is interpreted in $\mc{L}^2(f_{1,\frac{3}{2},\frac{1}{2}})$. Therefore,
\begin{align*}
L_N  = \sum_\lambda \widehat{\eta}(\lambda)^T h \left(\frac{1}{\sqrt{N}} \sum_{i=1}^N \chi_\lambda(x_i)\right) + \frac{1}{2} h^T \left( \frac{1}{N}\sum_{i=1}^N \nabla_\theta^2 \ell (x_i\mid \theta_0) \right)h + o_p(1).
\end{align*}
As $N\rightarrow \infty$, using Law of Large Numbers and Central Limit Theorem, one has
\begin{align*}
\frac{1}{N}\sum_{i=1}^N \nabla_\theta^2 \ell (x_i\mid \theta_0)  \rightarrow -\mc{I}(\theta_0)\quad \text{and} \quad
\frac{1}{\sqrt{N}} \sum_{i=1}^N \chi_\lambda(x_i)  \rightarrow Z_\lambda,
\end{align*}
where $Z_\lambda$ are independent standard normal variables. The joint limiting distribution of $(T_z^{(N)},L_N)$ is
\begin{align*}
(T_z^{(N)},L_N) \rightarrow (T_z, L),
\end{align*}
where
\begin{align}\label{Def:LimitJointLocalEigen}
T_z = \sum_{\lambda\neq 0} z^{|\lambda|} Z_\lambda^2 \quad \text{and} \quad L = \sum_\lambda (\widehat{\eta}(\lambda)^T h) Z_\lambda -  \frac{1}{2} h^T \mc{I}(\theta_0) h .
\end{align}
Since $f(\cdot\mid \theta)$ is Q.M.D., Le Cam's third lemma \cite[Theorem 12.3.3]{lehmann2006testing} implies that the limiting distribution of $T_z^{(N)}$ under $f(\cdot \mid \theta_0 + h/\sqrt{N})$ is given by the following characteristic function
\begin{align*}
\mb{E}_h e^{i t T_z} = \mb{E}_0 [e^{i t T_z} e^L] .
\end{align*}
Using (\ref{Def:LimitJointLocalEigen}) and the fact that $\{Z_\lambda\}$ are independent standard normal variables, one has
\begin{align*}
\mb{E}_h e^{i t T_z}  &= e^{-  \frac{1}{2} h^T \mc{I}(\theta_0) h } \prod_\lambda \mb{E}_0 e^{itz^{|\lambda|}Z_\lambda^2 + (\widehat{\eta}(\lambda)^T h) Z_\lambda }\\
&= e^{ \frac{1}{2}[ -h^T \mc{I}(\theta_0) h +\sum_\lambda (\widehat{\eta}(\lambda)^T h)^2]} \prod_\lambda \mb{E}_0 e^{itz^{|\lambda|}Z_\lambda^2 + (\widehat{\eta}(\lambda)^T h) Z_\lambda -\frac{1}{2} (\widehat{\eta}(\lambda)^T h)^2  } \\
&= e^{ \frac{1}{2}[ -h^T \mc{I}(\theta_0) h +\sum_\lambda (\widehat{\eta}(\lambda)^T h)^2]} \prod_\lambda  \int_{\mb{R}} e^{itz^{|\lambda|}Z_\lambda^2 + (\widehat{\eta}(\lambda)^T h) Z_\lambda -\frac{1}{2}  (\widehat{\eta}(\lambda)^T h)^2 } \frac{e^{\frac{-1}{2} Z_\lambda^2}}{\sqrt{2\pi}}  d Z_\lambda \\
&= e^{ \frac{1}{2}[ -h^T \mc{I}(\theta_0) h +\sum_\lambda (\widehat{\eta}(\lambda)^T h)^2]} \prod_\lambda  \int_{\mb{R}} e^{itz^{|\lambda|}Z_\lambda^2  } \frac{e^{\frac{-1}{2} \left( Z_\lambda^2 -2 (\widehat{\eta}(\lambda)^T h) Z_\lambda +  (\widehat{\eta}(\lambda)^T h)^2 \right)}}{\sqrt{2\pi}} d Z_\lambda \\
&= e^{ \frac{1}{2}[ -h^T \mc{I}(\theta_0) h +\sum_\lambda (\widehat{\eta}(\lambda)^T h)^2]} \prod_\lambda  \int_{\mb{R}} e^{itz^{|\lambda|}Z_\lambda^2  } \frac{e^{\frac{-1}{2} \left( Z_\lambda -  (\widehat{\eta}(\lambda)^T h) \right)^2}}{\sqrt{2\pi}} d Z_\lambda \\
&= e^{ \frac{1}{2}[ -h^T \mc{I}(\theta_0) h +\sum_\lambda (\widehat{\eta}(\lambda)^T h)^2]} \prod_\lambda \mb{E}e^{it z^{|\lambda|} U_\lambda}\\
&= e^{ \frac{1}{2}[ -h^T \mc{I}(\theta_0) h +\sum_\lambda (\widehat{\eta}(\lambda)^T h)^2]} \mb{E}e^{it\sum_\lambda z^{|\lambda|} U_\lambda}\\
&= \mb{E}e^{it\sum_\lambda z^{|\lambda|} U_\lambda}.
\end{align*}
where $U_\lambda = (Z_\lambda + \widehat{\eta}(\lambda)^T h)^2 \sim \chi_1^2\left((\widehat{\eta}(\lambda)^T h)^2\right)$ is a non-central chi-square variable on one degree of freedom with non-centrality parameter equal to $(\widehat{\eta}(\lambda)^T h)^2/2$.
The last step holds because $h^T \mc{I}(\theta_0) h  = \sum_\lambda (\widehat{\eta}(\lambda)^T h)^2$.

Thus, the limiting distribution of $T_z^{(N)}$ under the alternative $\theta_0+h/\sqrt{N}$ is
\begin{align}\label{Eq:LocalAlternativeTestStatisticsExpansionEigen}
T_z \sim \sum_{\lambda\neq 0} z^{|\lambda |} U_\lambda.
\end{align}
where $U_\lambda$ are as above.

The following proposition is an immediate consequence of the argument above.
\begin{prop}\label{Prop:AymptoticLocalPowerEigenTest} Let $c_{z,1-\alpha}$ be the asymptotic rejection threshold for $T_z^{(N)}$. That is, using Proposition \ref{NullAsymptoticDist},
\begin{align*}
\mb{P}\left(\sum_{k=1}^\infty z^k \chi^2_{p(n,k)} > c_{z,1-\alpha}\right) = \alpha,
\end{align*}
where $p(n,k)$ is the number of partitions of $k$ into at most $n$ parts and the chi-square variables are independent. 
The asymptotic power under the local alternative $\theta_0+h/\sqrt{N}$ is
\begin{align*}
\bm{\beta}(h) = \mb{P}\left(\sum_{\lambda\neq 0} z^{|\lambda |} U_\lambda > c_{z,1-\alpha}\right),
\end{align*}
for $U_\lambda$ as in (\ref{Eq:LocalAlternativeTestStatisticsExpansionEigen}).
\end{prop}
\begin{example} For $\theta \in \mb{R}$ let $f(x\mid \theta) \propto \exp(\theta \Tr (x))$. Then, $\eta(x \mid 0) = \Tr(x)$ and $\widehat{\eta}(\lambda) = 0$ for $\lambda \neq (1)$. The local power under $\theta/\sqrt{N}$ is
\begin{align*}
\bm{\beta}(\theta) = \mb{P}\left(z\ \chi_1^2(\theta^2) + \sum_{k=2}^\infty z^k \chi^2_{p(n,k)} > c_{z,1-\alpha}\right).
\end{align*}
\end{example}
The results of this section can be extended to families with infinite-dimensional parameter space, under mild regularity conditions. See Example 7.13 in the supplementary material \cite{sepehri2017supplement} for an example. 

%\subsection{Limitations of non-parametric tests of goodness-of-fit against local alternatives}
\subsection{Global asymptotic power function against local alternatives}
It is well-known that any test of goodness-of-fit exhibits poor power against local (contiguous) alternatives, except possibly in a finite number of directions. This section presents some results of this nature for $T_z^{(N)}$. For more details and statements in full generality see \citet{janssen1995principal}.

\subsubsection{Spectral decomposition of the power function}\label{Sec:PrincipalCompDecom}
Consider the standard local hypothesis setup.
For an arbitrary non-parametric unbiased test $\phi$ in the limiting Gaussian shift experiment, \citet{janssen1995principal} has shown that the curvature of the power function admits a principal component decomposition. 

Focus on $T_z^{(N)}$. Using the notation of section \ref{Sec:LocalPowerEigTEstFINITEDim}, the asymptotic power against the local alternative $H_1 : \theta = \theta_0 + h/\sqrt{N}$ is given by Proposition \ref{Prop:AymptoticLocalPowerEigenTest} as
\begin{align*}
\bm{\beta}(h) = \mb{P}\left(\sum_{\lambda\neq 0} z^{|\lambda |} U_\lambda > c_{z,1-\alpha}\right).
\end{align*}
The rejection cutoff $c_{z,1-\alpha}$ is such that $\bm{\beta}(\bm{0}) = \alpha$.
Using Theorem 1 in \citet{beran1975tail} one has the following second order Taylor expansion of $\bm{\beta}(th)$ around $t=0$
\begin{align*}
\bm{\beta}(t\cdot h) = \alpha + \frac{t^2}{2} \sum_\lambda (\widehat{\eta}(\lambda)^T h)^2 \left[ G_\lambda (c_{z,1-\alpha})- \alpha\right] +o(t^2),
\end{align*}
where $G_{\lambda}(x) = \mb{P}\left(\sum_{\mu} z^{|\mu|} V_\mu >x\right)$, and $V_\mu$ is a $\chi^2_1$ for $\mu \neq \lambda$, and $V_\lambda$ is a $\chi^2_3$ random variable.
Therefore, the curvature of the power function around $t = 0$ is
\begin{align*}
a(h) = \langle T(h), h \rangle,
\end{align*}
for the positive-definite bi-linear operator
\begin{align*}
T = \sum_\lambda \left[ G_\lambda (c_{z,1-\alpha})- \alpha \right]  \widehat{\eta}(\lambda) \widehat{\eta}(\lambda)^T.
\end{align*}
This readily gives a principal decomposition of the curvature, with principal components $\{ \widehat{\eta}(\lambda) \widehat{\eta}(\lambda)^T\}$ and eigenvalues $G_\lambda (c_{z,1-\alpha})- \alpha\ge 0$. For a fixed $z$ and $\alpha$, $G_\lambda (c_{z,1-\alpha})- \alpha$ is a decreasing function of $|\lambda|$. Thus, the highest gain in power is against those alternatives that put most of the load on principal components for smaller $|\lambda|$.
Theorem 2.1 in \citet{janssen1995principal} implies that $T$ is a Hilbert-Schmidt operator and $\|T\|^2 < 2\alpha(1-\alpha)$. This implies that any test performs poor against all alternatives except for a finite dimensional space. More details and various other statements are given in section 7 of the supplementary material \cite{sepehri2017supplement}. Local asymptotic relative efficiency and explicit bounds on the dimension of the subspace against which $T_z^{(N)}$ has power are also considered in section 7.3.3 of the supplementary material  \cite{sepehri2017supplement}.

\subsection{Asymptotic admissibility}
This section argues that the new tests of this paper are asymptotically admissible in the following sense. As discussed in section \ref{Sec:LocalPowerEigTEstFINITEDim}, there is a limiting hypothesis testing problem that captures the asymptotic properties of $T_z^{(N)}$ in the local hypothesis testing problem. That is, the limiting Gaussian process $\{Z_\lambda\}$, where under the limiting distribution corresponding to $h\in \Omega$, $Z_\lambda \sim \mc{N}(  \widehat{\eta}(\lambda)^T h,1)$. The problem is to test $H_0: h=0$. The limiting test statistic is $T_z = \sum_{\lambda\neq 0} z^{|\lambda|} Z_\lambda^2$. The main result is based on the following definition and is presented below.
\begin{defin}[Asymptotic admissibility] The sequence of test statistics $T_z^{(N)}$ is called asymptotically admissible if the limiting test $T_z$ is admissible for the limiting hypothesis testing problem.
\end{defin}
\begin{coro} The limiting tests based on $T_z$ is admissible. Therefore, the tests based on $\{T_z^{(N)}\}$ is asymptotically admissible.
\end{coro}
\begin{proof}
The proof is based on the following result of \citet{birnbaum1955characterizations}.
\begin{lemma}[\citet{strasser1985mathematical}, Theorem 30.4] \label{Lemma:BirnBaum} Let $C\subset \mb{H}$ be a closed convex subset. Then,
\begin{align*}
\phi(x) = \begin{cases} 1 \quad \text{if} \quad x \notin  C,\\
0 \quad \text{if} \quad x \in C,
\end{cases}
\end{align*}
is admissible for the testing problem $h=0$ against $h\neq 0$ and is uniquely determined by its power function.
\end{lemma}
Note that the test based on $T_z$ rejects for $\{\bm{Z}\in \mb{R}^\infty \mid \sum_\lambda z^{|\lambda|} Z_\lambda^2 > c_{z,1-\alpha}\}$. The set
\begin{align*}
C_z = \{\bm{Z}\in \mb{R}^\infty \mid \sum_\lambda z^{|\lambda|} Z_\lambda^2 \le c_{z,1-\alpha}\}
\end{align*}
is clearly convex and closed. Thus the assertion follows from Lemma \ref{Lemma:BirnBaum}.
\end{proof}
A similar statement is true for $\{U_{z,q}^{(N)}\}$ and is omitted here.
\section{Other Compact Groups}\label{Sec:OtherGroups}
The compact classical groups fall into four general classes:
\begin{enumerate}
\item \emph{Type A:} $U(n)$ and $SL(n)$.
\item \emph{Type B:} $SO(2n+1)$.
\item \emph{Type C:} $Sp(2n)$.
\item \emph{Type D:} $SO(2n)$.
\end{enumerate}
For each type the analogous tests to $T_z^{(N)}$ and $U_{z,q}^{(N)}$ are introduced in this section; only the definitions and explicit formulas are provided.
Derivation of the asymptotic null and alternative distributions, and local power are identical to those for $SO(2n+1)$, hence omitted here. Note that each case requires particular facts and considerations from representation theory; details, derivations, and proofs are provided in section 2 of the supplementary material \cite{sepehri2017supplement}.
\subsection{The test based on the eigenvalues}
For the groups of type A, \citet{coram2003new} introduced a test which inspired the test based on $T_z^{(N)}$ of the present paper. The case of type B groups ($SO(2n+1)$) was discussed in section \ref{Sec:EigenTest}. Type C and D are discussed below.
\subsubsection{Type C}
For $0<z<1$ define the test statistics $T_{C,z}^{(N)}$ as 
\begin{align*}
T_{C,z}^{(N)} = \frac{1}{N} \sum_{i=1}^N \sum_{j=1}^N K^C_z(g_i,g_j),
\end{align*}
where $K_z^C$ is as in \ref{Def:KernelGeneral}. A closed from for $K_z^C (g,h)$ can be found using the Cauchy identity for the symplectic group \citep[Theorem 2.7]{sepehri2017supplement} as follows:
\begin{align*}
K_z^C (g,h)  = \frac{(1-z^2)^m \det\left( \frac{1}{(1-z x_i y_j)(1-z x_i^{-1} y_j)(1-z x_i y_j^{-1})(1-z x_i^{-1} y_j^{-1})}\right)}{z^{\binom{m}{2}}\prod_{i<j} \left(y_i+y_i^{-1}-(y_j+y_j^{-1})\right) \prod_{i<j} \left(x_i+x_i^{-1}-(x_j+x_j^{-1})\right)}-1,
\end{align*}
where $\{x_i^\pm\}$ and $\{y_i^\pm\}$ are eigenvalues of $g$ and $h$ respectively.

\subsubsection{Type D}\label{Sec:TypeD-eigenTest} 
For $0<z<1$ define the test statistics $T_{D,z}^{(N)}$ as 
\begin{align*}
T_{D,z}^{(N)} =\frac{1}{N} \sum_{i=1}^N \sum_{j=1}^N K^D_z(g_i,g_j),
\end{align*}
where $K_z^D$ is as in \ref{Def:KernelGeneral}. A closed from for $K_z^D (g,h)$ is given by the Cauchy identity for $SO(2n)$ \citep[Theorem 2.9]{sepehri2017supplement} as follows:
\begin{align*}
K_z^D (g,h)  =  \frac{\det\left(\frac{1}{1-z x_i y_j}+\frac{1}{1-z x_i^{-1} y_j}+\frac{1}{1-z x_i y_j^{-1}}+\frac{1}{1-z x_i^{-1} y_j^{-1}}\right)}{z^{\binom{m}{2}}\prod_{i<j} \left(y_i+y_i^{-1}-(y_j+y_j^{-1})\right) \prod_{i<j} \left(x_i+x_i^{-1}-(x_j+x_j^{-1})\right)}-1,
\end{align*}
where $\{x_i^\pm\}$ and $\{y_i^\pm\}$ are eigenvalues of $g$ and $h$ respectively.

\subsection{The test beyond the eigenvalues} A test similar to $U_{z,q}^{(N)}$ can be constructed for all compact groups. With abuse of notation, these tests all are denoted by $U_{z,q}^{(N)}$. The case of $SO(2n+1)$ is already discussed in section \ref{Sec:TestBeyondEigenValues}. The other cases are considered in this section. Only the definitions and explicit formulas are presented here. A detailed derivation and required facts from representation theory, as well as the proofs, are provided in section 2 of the supplementary material \cite{sepehri2017supplement}.

\subsubsection{Type A} 
For $U(n)$, the test analogous to $U_{z,q}^{(N)}$ is defined as
\begin{align*}
U_{z,q}^{(N)} 
&=  \frac{1}{N} \sum_{i,j=1}^N \left(  \prod_{l,k=1}^{n} \frac{1}{1-zq^{l-1} y_k^{i,j}} -1\right),
\end{align*}
where $ y_1^{i,j}, \ldots,  y_n^{i,j}$ are the eigenvalues of $g_i^\ast g_j$.

\subsubsection{Type C} 
Definition for groups of type C is as follows
\begin{align*}
U_{z,q}^{(N)} &= \frac{1}{N}\sum_{k,l=1}^N  \left(\frac{\prod_{i<j}(1-z^2 q^{i+j-2})}{\prod_{i,j} (1-z q^{i-1} y_j^{k,l})(1-z q^{i-1} (y_j^{k,l})^{-1})} -1\right),
\end{align*}
where $\{y_j^{k,l}, (y_j^{k,l})^{-1}\mid j=1,\ldots,n \}$ are the eigenvalues of $g_k g_l^T$.

\subsubsection{Type D}
Definition for groups of type D is as follows
\begin{align*}
U_{z,q}^{(N)} &= \frac{1}{N}\sum_{k,l=1}^N  \left(\frac{\prod_{i\le j}(1-z^2 q^{i+j-2})}{\prod_{i,j} (1-z q^{i-1} y_j^{k,l})(1-z q^{i-1} (y_j^{k,l})^{-1})} -1\right),
\end{align*}
where $\{y_j^{k,l}, (y_j^{k,l})^{-1}\mid j=1,\ldots,n \}$ are the eigenvalues of $g_k g_l^T$.

\section{Discussion}
The current paper introduces and analyzes two new families of tests of uniformity on the compact classical groups. These tests are validated on two benchmark examples: the random walk of Kac and the products of random reflections. They exhibit satisfying agreement with the existing theory about the mixing-time of both random walks. The new tests, and several others, are applied to the new sampler of \citet{jones2011randomized}; all but one of the new tests failed to reject the null hypothesis of uniformity after any number of iterations of the new sampler. One of the new tests confirmed the prescribed number of steps to be used with the sampler in order to get approximately uniform outputs.

\section{Acknowledgments} I am greatly indebted to my doctoral advisor, Persi Diaconis, for suggesting the problem and his continuing guidance and support. I thank Daniel Bump for discussions on representation theory; Arun Ram for sharing his notes on the character theory and symmetric function theory; Sourav Chatterjee and David Siegmund for helpful comments and discussions; Matan Gavish for sharing his course report on numerical investigation of the mixing-time for a Markov chain sampler on the unitary group. I would like to express my special thanks to the associate editor and the anonymous referee because of their useful comments that improved this paper significantly. In particular, Remark 4.8 was the result of a suggestion by the associate editor.

\bibliography{TestingOnClassicalGroups}

%\end{spacing}

\end{document}